\documentclass[10pt, reqno]{amsart}
\usepackage[dvipsnames,usenames]{color}
\usepackage{hyperref}
\usepackage{graphicx}
\usepackage{epsfig}
\usepackage[latin1]{inputenc}
\usepackage{amsmath}
\usepackage{amsfonts}
\usepackage{amssymb}
\usepackage{amsthm}
\usepackage{amscd}
\usepackage{verbatim}
\usepackage{subfigure}
\usepackage{caption}
\usepackage{pinlabel}
\usepackage{stmaryrd}
\usepackage{enumerate}%, enumitem}
\usepackage{todonotes}
\usepackage{bm}
\usepackage{thmtools}
\usepackage{thm-restate}
\usepackage{lipsum}
\usepackage{setspace}
\usepackage{mathtools}
\usepackage[all]{xypic}
\usepackage[abs]{overpic}
\usepackage{color}
\usepackage[shortlabels]{enumitem}

\allowdisplaybreaks

\usepackage{tikz}
\usetikzlibrary{arrows}
\usetikzlibrary{decorations.pathreplacing}
\usepackage{verbatim}
\usetikzlibrary{cd}
\tikzset{taar/.style={double, double equal sign distance, -implies}}
\tikzset{amar/.style={->, dotted}}
\tikzset{dmar/.style={->, dashed}}
\tikzset{aar/.style={->, very thick}}

\newtheorem{theorem}{Theorem}[section]

\newtheorem{lemma}[theorem]{Lemma}
\newtheorem{proposition}[theorem]{Proposition}

\newtheorem{corollary}[theorem]{Corollary}

\newtheorem{conjecture}[theorem]{Conjecture}

\newtheorem{question}[theorem]{Question}

\theoremstyle{definition}
\newtheorem{definition}[theorem]{Definition}

\theoremstyle{remark}
\newtheorem{remark}[theorem]{Remark}

\def\F{\mathbb{F}}

\def\R{\mathbb{R}}
\def\Z{\mathbb{Z}}

\def\Im{\operatorname{Im}}

\def \gr {\operatorname{gr}}

\def \br {\operatorname{br}}
\def \b {\operatorname{b}}

\def\d{\partial}

\def\HF {\mathit{HF}}
\newcommand\HFhat{\widehat{\HF}}

\def\CFK{\mathit{CFK}}

\newcommand\HFKm{\mathit{HFK}^-}

\def\Tor{\operatorname{Tor}}

\newcommand{\U}{\mathcal{U}}
\newcommand{\V}{\mathcal{V}}

\newcommand{\lab}[1]{$\scriptstyle #1$}

\newcommand{\Ord}{\operatorname{Ord}}

\author[J.\ Hom]{Jennifer Hom}
\thanks{The first author was partially supported by NSF grants DMS-1552285 and DMS-2104144.}
\address {School of Mathematics, Georgia Institute of Technology, Atlanta, GA 30332}
\email{hom@math.gatech.edu}

\author[T. Lidman]{Tye Lidman}
\thanks {The second author was partially supported by NSF grants DMS-1709702 and DMS-2105469 and a Sloan Fellowship}
\address{Department of Mathematics, North Carolina State University, Raleigh, NC, 27607}
\email{tlid@math.ncsu.edu}

\author[J.\ Park]{JungHwan Park}
\address {Department of Mathematical Sciences, KAIST, Daejeon 34141, Republic of Korea}
\email{jungpark0817@kaist.ac.kr}
\thanks{The third author was partially supported by Samsung Science and Technology Foundation (SSTF-BA2102-02) and the POSCO TJ Park Science Fellowship.}

\numberwithin{equation}{section}

\title{Unknotting number and cabling}

\begin{document}

\begin{abstract} 
The unknotting number of knots is a difficult quantity to compute, and even its behavior under basic satelliting operations is not understood.  We establish a lower bound on the unknotting number of cable knots and iterated cable knots purely in terms of the winding number of the pattern.  The proof uses Alishahi-Eftekhary's bounds on unknotting number from knot Floer homology together with Hanselman-Watson's computation of the knot Floer homology of cables in terms of immersed curves in the punctured torus.    
\end{abstract}

\maketitle

\section{Introduction}

The unknotting number $u(K)$ of a knot $K$ is one of the simplest knot invariants, yet not much is known about its properties. For instance, we do not know how the unknotting number changes with respect to the connected sum. Scharlemann~\cite{Scharlemann:1985-1} proved that composite knots have unknotting number at least 2, and there are some other partial known results e.g.\ \cite{Yang:2008-1, AlishahiEftekharyunknotting}, but the problem remains very much open. 

The connected sum operation is one of the simplest satellite constructions, and we know just as little about most other satellite operations, apart from special cases like Whitehead doubling.  In this article, we study the behavior of the unknotting number under another simple satellite operation, \emph{cabling}. Let $K_{p,q}$ denote the $(p, q)$-cable of $K$, where $p$ denotes the longitudinal winding, and let $K_{p_1, q_1; p_2, q_2; \dots ; p_m, q_m}$ denote the iterated cable of $K$. We may assume that $p > 1$ and $p_i>1$ since $K_{p,q}$ is isotopic to $K_{-p,-q}$ and $K_{1,q}$ is isotopic to $K$.
%Without loss of generality, we may assume that $p > 1$ and $p_i>1$.

\begin{theorem}\label{thm:main}
If $K$ is a nontrivial knot, then $u(K_{p,q}) \geq p$ and $u(K_{p_1, q_1; p_2, q_2; \dots ; p_m, q_m}) \geq p_1 p_2 \dots p_m$.
\end{theorem}

We make a remark that it was previously shown by Scharlemann-Thompson~\cite[Corollary 3.3]{Scharlemann-Thompson:1988-1} that $u(K_{p,q}) \geq 2$ for any nontrivial knot $K$. In fact, they showed that a satellite knot of nonzero winding number in the pattern has unknotting number at least 2 (in particular, it recovers the result of Scharlemann~\cite{Scharlemann:1985-1} mentioned above). 

When $K$ is a torus knot $T_{p_0,q_0}$, we have the following stronger statement. Since $u(K)=u(-K)$ where $-K$ is the reverse of the mirror image of a knot $K$, we assume that  $T_{p_0,q_0}$ is a positive torus knot.  

\begin{theorem}\label{thm:iteratedtorus}
If $1<p_0 < q_0$ and $q_0 \neq 3$, then $u(T_{p_0, q_0; p_1, q_1; \dots ; p_m, q_m}) \geq p_0 p_1 \dots p_m-1$. 
\end{theorem}

The main tool we use to establish these lower bounds is the torsion order of a knot (see Definition~\ref{def:torsionorder}) coming from knot Floer homology~\cite{OSknots, Rasmussen-thesis}. The torsion order of a knot $K$ is denoted by $\Ord(K)$ and has many topological applications. In particular, Alishahi-Eftekhary~\cite[Theorem 1.1]{AlishahiEftekharyunknotting} (see also \cite{Zemke-funct}) showed that if $K$ is a knot, then
$$u(K) \geq  \Ord(K).$$

%\textcolor{red}{TL: Added this paragraph - I think an introduction should contain a sketch of the proof if it's easy to explain, but this could be put elsewhere in the intro.} \textcolor{purple}{JH: I'm for it.}  

We quickly sketch the idea of the proof.  The knot Floer homology of a knot can be recast in the immersed curves interpretation~\cite{HRW,HRW:2018} of bordered Floer homology~\cite{LOT}.  The bordered Floer homology of the exterior of a knot $K$ in $S^3$ can be thought of roughly as an immersed multicurve in the boundary of the knot exterior (with a single point removed).  The torsion order can be morally thought of as the maximal number of times this multicurve wraps consecutively in the direction of the meridian of the knot.  Hanselman-Watson~\cite{HW-cables} gives a direct topological formula for transforming the immersed curve for $K$ into the immersed curve for $K_{p,q}$, and we will show the winding number loosely translates into how much additional wrapping is introduced to the curves.

Using the above strategy, we will obtain the following more general bound on $\Ord(K)$. The inequality itself may be of independent interest, and a similar computation for a certain family of knots can be found in~\cite[Proposition 1.2]{Hom-Kang-Park:2020-1}.

\begin{proposition}\label{prop:main}
If $K$ is a nontrivial knot, then 
$$\Ord(K_{p_1, q_1; p_2, q_2; \dots ; p_m, q_m}) \geq \max \left\{ p_1  p_2 \dots  p_m \left(\Ord(K)-1\right)+1, p_1  p_2 \dots  p_m \right\}.$$
\end{proposition}

As above, when $K$ is a torus knot, or more generally an $L$-space knot, we have the following stronger statement. Note that $\Ord(K)=\Ord(-K)$, since the knot Floer homology of $-K$ is the dual of the knot Floer homology of $K$~\cite[Section 5.1]{Ozsvath-Szabo:2006-1}. 

\begin{proposition}\label{prop:iteratedtorus}
If $K$ is a nontrivial $L$-space knot and $K\neq T_{2,3}$, then $$\Ord(K_{p_1, q_1; p_2, q_2; \dots ; p_m, q_m}) \geq p_1 \dots p_m \left(\Ord(K)+1\right)-1.$$
Moreover, if the bridge index of $K$ is $\Ord(K)+1$, then the above inequality is an equality.
\end{proposition}

\subsection{More applications}\label{subsec:applications}

As mentioned earlier, the torsion order of a knot has many topological applications. We take advantage of this fact and further explore the consequences of Proposition~\ref{prop:main} and Proposition~\ref{prop:iteratedtorus}.  Throughout the article, we will always work in the smooth category.

A knot is \emph{slice} if it bounds a smoothly embedded disk in $B^4$. Note that the $(p,1)$-cable of a slice knot is also slice since there is a disk for $K_{p,1}$ constructed by taking $p$ parallel copies of a disk bounded by $K$ and attaching $p-1$ half-twisted bands to them. In this construction, each parallel copy of a disk for $K$ gives at least one local minima for the resulting disk for $K_{p,1}$. Surprisingly, we show that \emph{any} disk bounded by $K_{p,1}$ has at least $p+1$ local minima given that $K$ is a nontrivial slice knot. This is done by combining Proposition~\ref{prop:main} with a bound of Juh\'{a}sz-Miller-Zemke~\cite[Corollary 1.7]{Juhasz-Miller-Zemke:2020-1}.

\begin{corollary}\label{cor:fusionnumber}
If $K$ is a nontrivial slice knot, then any slice disk bounded by $K_{p_1, 1; p_2, 1; \dots ; p_m, 1}$ has at least $p_1 p_2 \dots p_m+1$ local minima with respect to the radial function on $B^4$ restricted to the disk.\end{corollary}

We make a remark that it was previously shown in~\cite[Theorem 1.1]{Hom-Kang-Park:2020-1} that if a nontrivial knot $K$ bounds a disk with 2 local minima, then any disk bounded by $K_{p, 1}$ has at least has $p+1$ local minima (in fact, they also show that there is a disk for $K_{p, 1}$ with $p+1$ local minima).   

We also provide applications  to concordance.  Two knots $K$ and $J$ are \emph{concordant} if $K\# -J$ is slice, and the set of concordance classes of knots form a group called the \emph{knot concordance group} $\mathcal{C}$ under the operation of connected sum. Every element in $\mathcal{C}$ is represented by infinitely many knots, and it is natural to ask which knot is the simplest representative of the class. There are many natural ways to measure the complexity of a knot, and it turns out that the torus knots, and more generally algebraic knots, those that are one component links of isolated singularities of complex algebraic plane curves, have the minimal complexity among all the knots that are concordant to them with respect to various geometric knot invariants.  

For instance, the resolution of the local Thom conjecture by Kronheimer-Mrowka~\cite[Corollary 1.3]{Kronheimer-Mrowka:1993-1} implies that each algebraic knot minimizes both the 3-genus and the unknotting number in this sense. Moreover, Juh\'{a}sz-Miller-Zemke~\cite[Corollary 1.10]{Juhasz-Miller-Zemke:2020-1} proved that this is also the case for torus knots with respect to the braid index and the bridge index (see also~\cite[Theorem 1.3]{Feller-Krcatovich:2017-1}). This was done by finding a lower bound for the braid index and the bridge index in terms of the torsion order~\cite[Corollary 1.8]{Juhasz-Miller-Zemke:2020-1}. Using Propostion~\ref{prop:iteratedtorus}, we extend their result to a large family of iterated torus knots including all algebraic knots. For a knot $K$, its bridge index is denoted by $\br(K)$ and its braid index is denoted by $\b(K)$. Again, since $\b(K)=\b(-K)$ and $\br(K)=\br(-K)$, we assume that  $T_{p_0,q_0}$ is a positive torus knot.

\begin{corollary}\label{cor:bridgenumber}
Let $K$ be either an algebraic knot or an iterated torus knot $T_{p_0, q_0; p_1, q_1; \dots ; p_m, q_m}$ where $T_{p_0, q_0} \neq T_{2,3}$. If $J$ is concordant to $K$, then $\b(J) \geq \b(K)$ and $\br(J) \geq \br(K)$.
\end{corollary}

We believe that the statement should also be true even when $T_{p_0, q_0} = T_{2,3}$, though one would need to use a different method than ours to prove this since there are examples where the torsion order does not give an optimal bound for the bridge number (e.g.,\ $\Ord(T_{2,3; 2,3})=2$ whereas $\br(T_{2,3; 2,3})-1=3$).

We provide one more 4-dimensional application.  Recall that the \emph{4-genus} of a knot $K$, denoted by $g_4(K)$, is defined to be the minimal genus of an oriented smooth surface properly embedded in $B^4$ bounded by $K$. It is well known that the 4-genus gives a lower bound for the unknotting number, and the two invariants coincide for the closures of positive braids~\cite{Stoimenow:2003-1}. Hence for a large family of iterated torus knots, including algebraic knots, the 4-genus is equal to the unknotting number. This is not true for all iterated torus knots. In fact, we show the following.

\begin{corollary}\label{cor:genusunknotting}
There exists a family of iterated torus knots $\{K_i\}$, such that the gap between $u(K_i)$ and $g_4(K_i)$ is arbitrarily large.\end{corollary}

Lastly, we give an application to the untwisting number, a generalization of the unknotting number. A \emph{null-homologous twist} is an operation on a knot, where we perform $1$ or $-1$ surgery on an unknotted curve that is null-homolgous in the knot complement. The \emph{untwisting number} of a knot $K$, denoted by $u_t(K)$, is defined to be the minimum number of null-homologous twists requried to convert $K$ to the unknot (for related work see e.g.\ \cite{Mathieu-Domergue:1988-1, Ince:2016-1, Ince:2017-1, Livingston:2021-1, McCoy:2021-1, McCoy:2021-2}). By definition, it is clear that for any nontrivial knot $K$, we have $$u(K) \geq u_t(K) \geq 1.$$ In \cite[Theorem 1.3]{Ince:2016-1}, it was shown that if $K$ is an unknotting number one knot with $\tau(K) >0$, where $\tau$ is the concordance invariant of \cite{Ozsvath-Szabo:2003-1}, then $u(K_{p,1}) \geq p$ and $u_t(K_{p,1}) =1$. The fact that $\tau$ gives a lower bound for the unknotting number~\cite{Ozsvath-Szabo:2003-1} and a cabling formula for $\tau$ of~\cite[Theorem 1]{Hom:2014-1b} imply that $u(K_{p,1}) \geq p$. The equality $u_t(K_{p,1}) =1$ follows easily and is left to the reader. Theorem~\ref{thm:main} directly implies that the above result generalizes to all unknotting number one knots:

\begin{corollary}\label{cor:untwisting}
If $K$ is an unknotting number one knot, then $u(K_{p,1}) \geq p$ and $u_t(K_{p,1}) =1$.\qed \end{corollary}

\subsection{Some open problems}\label{subsec:questions}
We end this section with a few open problems on the unknotting number. Given a nontrivial knot $K$, one easily obtains a crude upper bound on $u(K_{p,q})$ as follows. First, choose an unknotting sequence for $K$.  For each crossing change for $K$, perform the corresponding $p^2$ crossing changes on $K_{p,q}$. Then the resulting knot is a torus knot $T_{p,q'}$  where $q'$ depends on the unknotting number of $K$ and also the sign of the crossing changes that were performed to unknot $K$. By further performing the minimal number of crossing changes needed to turn $T_{p,q'}$ into the unknot,  we obtain a bound.

%Hence we conclude that $u(K_{2,1}) \leq 6$ for any unknotting number one knot $K$. 

For instance, if $K$ is a knot that can be turned into the unknot by changing a positive crossing to a negative crossing (e.g.,\ $T_{2,3}$), then this crude bound tells us that $u(K_{2,1}) \leq 5$. Similarly, if $K$ turns into the unknot by changing a negative crossing to a positive crossing, then we have $u(K_{2,1}) \leq 6$. Unfortunately, in either case these bounds are far from sharp. For instance, it is a straightforward exercise to find a sequence of crossing changes for $T_{2,3;2,1}$ and show that $u(T_{2,3;2,1})\leq 3$. Combined with Theorem~\ref{thm:main} (or \cite{Scharlemann-Thompson:1988-1}), we see that $u(T_{2,3;2,1})$ is either 2 or 3. To our surprise, it seems that the unknotting number of $T_{2,3;2,1}$ is not known. We ask the following questions.

\begin{question}\label{qu:questions}Is there a nontrivial knot $K$ where we know the value of $u(K_{2,1})$? More generally, is there a formula for the unknotting number of cabled knots?
\end{question}

Finally, we remark that it is likely that the ideas in this paper can be generalized to other satellite operations.  In particular, work of Chen~\cite{chen2019knot} describes the knot Floer homology of satellite knots with (1,1)-patterns in the solid torus in terms of immersed curves similar to the work of Hanselman-Watson~\cite{HW-cables} for cables used here.  It would be interesting to apply similar techniques to obtain bounds on the unknotting number of such satellite knots.  Furthermore, it is natural to ask how the winding number of a pattern governs the unknotting number of a satellite knot.  In line with Scharlemann-Thompson's result that a nontrivial satellite of nonzero winding number has unknotting number at least 2 \cite{Scharlemann-Thompson:1988-1}, we pose     
\begin{conjecture}\label{conj:winding}
Suppose that $K$ is a satellite of a nontrivial knot, where the pattern has winding number $p$.  Then, $u(K) \geq p+1$.
\end{conjecture}

Using Theorem~\ref{thm:main}, we are able to give further evidence for Conjecture~\ref{conj:winding} being true using Theorem~\ref{thm:main}.  An analogous statement for the Levine-Tristram signature~\cite{Levine:1969-1,Tristram:1969-1} can be proved using~\cite[Theorem 1]{Litherland:1984-1}. 

\begin{corollary}
Let $P$ be a nontrivial pattern with winding number $p \geq 1$.  There exists a constant $C_P$ such that for any companion $K$ with $\Ord(K) \geq C_P$, we have $u(P(K)) \geq p+1$.  
\end{corollary}
\begin{proof}
Let $C_{p,1}$ be the $(p,1)$-cabling pattern in the solid torus. Choose a sequence of $c$ crossing changes in the solid torus taking $C_{p,1}$ to $P$.  Therefore, $u(K_{p,1}) \leq u(P(K)) + c$.  Moreover, if $u(K_{p,1}) \geq c + p + 1$, then $u(P(K)) \geq p + 1$.  The result now follows from Proposition~\ref{prop:main}.
\end{proof}

\section*{Organization}
In Section~\ref{sec:background}, we review and develop the relevant connections between knot Floer complexes and the immersed curve machinery of Hanselman-J. Rasmussen-Watson.  In Section~\ref{sec:cables}, we use the immersed curve description of cabling due to Hanselman-Watson to prove Propositions~\ref{prop:main} and \ref{prop:iteratedtorus}. Lastly, in Section~\ref{sec:applications}, we prove the applications of the propositions including Theorems~\ref{thm:main} and \ref{thm:iteratedtorus}.

\section*{Acknowledgements}
 We would like to thank Peter Feller for helpful conversations.

\section{Knot-like complexes}\label{sec:background}
We will assume that the reader is familiar with knot Floer homology \cite{OSknots, Rasmussen-thesis}. We will view our knot Floer complexes as chain complexes over the ring $\F[\U, \V]$ or $\F[\U, \V]/(\U \V)$. See \cite[Section 3]{HomPCMI} for an expository overview. Chain complexes over the latter ring have an interpretation as immersed curves in the punctured torus. The goal of this section is to describe this relationship, which is described in \cite[Section 4]{HRW:2018} (see also \cite{KWZ}), and to prove a technical result obstructing certain types of curves from being realized by such chain complexes.

\subsection{Modules over $\F[\U, \V]/(\U \V)$ and immersed curves}\label{sec:complextocurve}

We will consider the bigraded ring $\F[\U, \V]$, where $\gr(\U) = (-2, 0)$ and $\gr(\V) = (0, -2)$. We refer to the first component of $\gr$ as $\gr_\U$ and the second component as $\gr_\V$.

Every chain complex $C$ over $\F[\U, \V]$ is homotopy equivalent to one that is \emph{reduced}; that is, that $\d x = 0 \mod (\U, \V)$ for all $x \in C$. Without loss of generality, we will assume throughout that our chain complexes are reduced.

\begin{definition}\label{def:knot-like}
A \emph{knot-like complex} $C$ over $\F[\U, \V]$ is a free, finitely generated, bigraded chain complex over $\F[\U, \V]$ such that
\begin{enumerate}
	\item \label{it:knot-like1} $H_*(C/\U) / \V \textup{-torsion} \cong \F[\V]$ where $\gr_\U(1) = 0$,
	\item \label{it:knot-like2}  $H_*(C/\V) / \U \textup{-torsion} \cong \F[\U]$ where $\gr_\V(1) = 0$.
\end{enumerate}
\end{definition}

\begin{remark}
Similarly, we may define a \emph{knot-like complex} $C$ over the ring $\F[\U, \V]/(\U \V)$ (see \cite[Definition 3.1]{DHSTmore}). In the present paper, we work over both rings, as Proposition \ref{prop:d2neq0} below relies on the existence of a nonzero term in the image of $\U\V$.
\end{remark}
%In fact, in \cite[Definition 3.1]{DHSTmore}, knot-like complexes were defined over the ring $\F[\U, \V]/(\U \V)$. 

Knot-like complexes arise as the knot Floer complexes of knots in $S^3$. That is, if $K$ is a knot in $S^3$, then the knot Floer complex $\CFK(K)$ is a knot-like complex. See, for example, \cite[Section 3.2]{HomPCMI}.

We let $H^-(C)$ denote $H_*(C/\V)$. Note that when $C= \CFK(K)$, we have that $H^-(C) = \HFKm(K)$. The Alexander grading is given by $\frac{1}{2}( \gr_\U - \gr_\V)$.  

\begin{definition}\label{def:torsionorder}
The \emph{torsion order} of a knot-like complex $C$ is
\[ \Ord(C) = \min \{ n \mid \U^n \cdot \Tor_\U H^-(C) = 0 \}. \]
The \emph{torsion order} of a knot $K$ is $\Ord(K) = \Ord(\CFK(K))$, where $\CFK(K)$ is the knot Floer complex of $K$.
\end{definition}

\begin{theorem}[Theorem 1.1 of \cite{AlishahiEftekharyunknotting}]
The torsion order of a knot $K$ provides a lower bound on the unknotting number of $K$:
\begin{equation}\label{eq:unknotting}
u(K) \geq 	\Ord(K).
\end{equation}
\end{theorem}

Let $\{z_i\}$ be a basis over $\F[\U, \V]$ for a knot-like complex $C$. Let $\langle \d z_i, z_j \rangle$ denote the coefficient of $z_j$ in $\d z_i$. If $\langle \d z_i, z_j \rangle = \U^n$, then we say that there is a \emph{horizontal arrow of length $n$} from $z_i$ to $z_j$. If $\langle \d z_i, z_j \rangle = \V^n$, then we say that there is a \emph{vertical arrow of length $n$} from $z_i$ to $z_j$. If $\langle \d z_i, z_j \rangle \in \Im (\U \V)$, then we say there is a \emph{diagonal arrow} from $z_i$ to $z_j$. (Observe that $\langle \d z_i, z_j \rangle$ must be a monomial in $\U$ and $\V$ for grading reasons.)

We now consider certain particularly nice bases for knot-like complexes. A basis $\{x_i\}_{i=0}^{2N}$ for a knot-like complex over $\F[\U,\V]$ is called \emph{horizontally simplified} if $\d x_i = 0 \mod (\V)$ for $i$ even and $\d x_i = \U^{n_i} x_{i+1} \mod (\V)$ for $i$ odd.  See Figure~\ref{fig:horizontally-simplified} for a graphical depiction.  
%We say that there is a \emph{horizontal arrow of length $n_i$} from $x_i$ to $x_{i+1}$.

%\begin{remark}\label{rk:dxixj}
%Note that if $j$ is odd, then $\langle \d x_i, x_j \rangle = 0 \mod (\V)$ for all $i$, where $\langle \d x_i, x_j \rangle$ denotes the coefficient of $x_j$ in $\d x_i$.
%\end{remark}

%Let $\langle \d x_i, x_j \rangle$ denote the coefficient of $x_j$ in $\d x_i$. If $\langle \d x_i, x_j \rangle = \U^m$, then we say that there is a \emph{horizontal arrow of length $m$} from $x_i$ to $x_j$.

\begin{figure}[htb!]
\begin{center}
\begin{tikzpicture}[scale=1]
	%\draw[step=1, black!30!white, very thin] (-1.5, -4.5) grid (4.5, 4.5);

	\filldraw (3.5, 3.8) circle (2pt) node[label=right:{\lab{x_0}}] (a) {};
	\filldraw (3.5, 3.5) circle (2pt) node[label=right:{\lab{x_1}}] (b) {};
	\filldraw (0.5, 3.5) circle (2pt) node[label=left :{\lab{\U^3 x_2}}] (c) {};
	\filldraw (3.5, 3.2) circle (2pt) node[label= right:{\lab{x_3}}] (d) {};
	\filldraw (1.5, 3.2) circle (2pt) node[label= left:{\lab{\U^2 x_4}}] (f) {};

	\draw[->] (b) to (c);
	\draw[->] (d) to (f);

\end{tikzpicture}
\caption{A graphical depiction of a horizontally simplified basis modulo $(\V)$.}
\label{fig:horizontally-simplified}
\end{center}
\end{figure}
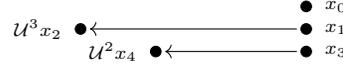

Similarly, we say that a basis $\{y_i\}_{i=0}^{2N}$ for a knot-like complex over $\F[\U,\V]$ is \emph{vertically simplified} if $\d y_i = 0 \mod (\U)$ for $i$ even and $\d y_i = \V^{m_i} y_{i+1} \mod (\U)$ for $i$ odd. 
A horizontally simplified basis and a vertically simplified basis for a knot-like complex over $\F[\U, \V]/(\U \V)$ is defined in a similar way.

%If $\langle \d y_i, y_j \rangle = \V^m$, then we say that there is a \emph{vertical arrow of length $m$} from $y_i$ to $y_j$.

\begin{remark}\label{rk:dxixj}
Note that if $\{y_i\}$ is a vertically simplified basis and $j$ is odd, then $\langle \d y_i, y_j \rangle = 0 \mod (\U)$ for all $i$.
\end{remark}

We can always choose a vertically simplified basis for $C$, or, if we prefer, a horizontally simplified basis for $C$; see, for example, \cite[Proposition 11.57]{LOT}. 

In many situations, it suffices to work with chain complexes over $\F[\U, \V]/(\U \V)$ rather than $\F[\U, \V]$. We now describe the relationship between chain complexes over $\F[\U, \V]/(\U \V)$ and immersed curves. Our immersed curves will sit in $(\R / \Z) \times \R$ with punctures at each lattice point $\left(0,n+\frac{1}{2}\right)$. The punctures may be thought of as pegs around which the immersed curves wind. For convenience, we identify $(\R / \Z) \times \R$ with $\left(\left[-\frac{1}{2}, \frac{1}{2}\right]/ -\frac{1}{2} \sim \frac{1}{2}\right) \times \R$ and work with the latter.

%We consider a fundamental domain $\left[-\frac{1}{2}, \frac{1}{2}\right) \times \R$ of $(\R / \Z) \times \R$.  

Let $\{x_i\}$ be a horizontally simplified basis for a knot-like complex over $\F[\U,\V]/(\U \V)$. Each basis element $x_i$ corresponds to a point on the vertical line $\{ 0 \} \times \R$; we abuse notation and also denote this point by $x_i$. The vertical coordinate of $x_i$ is given by the Alexander grading of $x_i$. A horizontal arrow of length $n_i$ from $x_i$ to $x_{i+1}$, $i$ odd, gives rise to an arc of length $n_i$ from $x_i$ to $x_{i+1}$ to the right of the vertical line $\{ 0 \} \times \R$, where $x_i$ is $n_i$ units below $x_{i+1}$. We call such an arc a \emph{right arc of length $n_i$}. We also include an arc from $x_0$ to $(\frac{1}{2}, 0)$, which we call the \emph{essential right arc}.

Similarly, if $\{y_i\}$ is a vertically simplified basis for a knot-like complex over $\F[\U,\V]/(\U \V)$, each basis element $y_i$ corresponds to a point on the vertical line $\{ 0 \} \times \R$; we again abuse notation and denote the point by $y_i$. The vertical coordinate of $y_i$ is again given by the Alexander grading of $y_i$. A vertical arrow of length $m_i$ from $y_i$ to $y_{i+1}$, $i$ odd, gives rise to an arc of length $m_i$ from $y_i$ to $y_{i+1}$ to the left of the vertical line $\{ 0 \} \times \R$, where now $y_i$ is $m_i$ units above $y_{i+1}$. We call such an arc a \emph{left arc of length $m_i$}. We also include an arc from $y_0$ to $(-\frac{1}{2}, 0)$, which we call the \emph{essential left arc}.  

Thus, given a knot-like complex with a simultaneously horizontally and vertically simplified basis, it is straightforward to find an associated immersed curve, which will be a concatenation of right and left arcs of various lengths, together with the essential arcs. See Figure \ref{fig:compextocurve} for some examples.

\begin{figure}[htb!]
\begin{center}
\begin{tabular}[c]{ccc}
\subfigure[]{
\begin{tikzpicture}[scale=0.7]
	\draw[step=1, black!30!white, very thin] (-2.5, -1.5) grid (4.5, 4.5);

	\filldraw (3.5, 3.5) circle (2pt) node[label=above:{\lab{a}}] (b) {};
	\filldraw (-1.5, 3.5) circle (2pt) node[label=above :{\lab{\U^5 b}}] (c) {};
	\filldraw (2.5, 2.5) circle (2pt) node[label=below left:{\lab{\U \V d}}] (d) {};
	\filldraw (3.5, 0.5) circle (2pt) node[label=right:{\lab{\V^3 f}}] (e) {};
	\filldraw (-1.5, 2.5) circle (2pt) node[label=left:{\lab{\U^5 \V c}}] (f) {};
	\filldraw (2.5, 0.5) circle (2pt) node[label=below:{\lab{\U \V^3 e}}] (g) {};

	\draw[->] (b) to (c);
	\draw[->] (b) to (d);
	\draw[->] (b) to (e);
	\draw[->] (c) to (f);
	\draw[->] (d) to (f);
	\draw[->] (d) to (g);
	\draw[->] (e) to (g);

\end{tikzpicture}
}&
\subfigure[]{
\labellist
	\pinlabel {\lab{b}} at  55 160
	\pinlabel {\lab{c}} at  55 140
	\pinlabel {\lab{a}} at  55 80
	\pinlabel {\lab{d}} at  65 74
	\pinlabel {\lab{e}} at  65 50
	\pinlabel {\lab{f}}  at  55 30
\endlabellist
\includegraphics{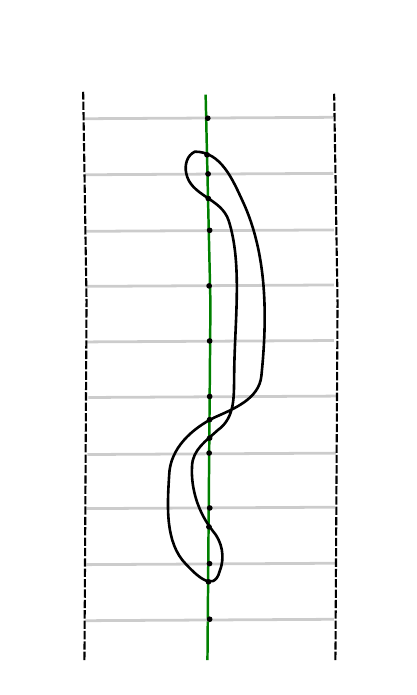}
}\\
\subfigure[]{
\begin{tikzpicture}[scale=0.7]
	\draw[step=1, black!30!white, very thin] (-0.5, -0.5) grid (5.5, 5.5);

	\filldraw (0.5, 4.5) circle (2pt) node[label=above:{\lab{\U^4 a}}] (a) {};
	\filldraw (1.5, 4.5) circle (2pt) node[label=above:{\lab{\U^3 b}}] (b) {};
	\filldraw (1.5, 1.5) circle (2pt) node[label=left:{\lab{\U^3 \V^3 c}}] (c) {};
	\filldraw (4.5, 1.5) circle (2pt) node[label=above:{\lab{\V^3 d}}] (d) {};
	\filldraw (4.5, 0.5) circle (2pt) node[label=right:{\lab{\V^4 e}}] (e) {};

	\draw[->] (b) to (a);
	\draw[->] (b) to (c);
	\draw[->] (d) to (c);
	\draw[->] (d) to (e);

\end{tikzpicture}
}&
\subfigure[]{
\labellist
	\pinlabel {\lab{a}} at  55 160
	\pinlabel {\lab{b}} at  65 140
	\pinlabel {\lab{c}} at  65 95
	\pinlabel {\lab{d}} at  55 47
	\pinlabel {\lab{e}} at  55 28

	\pinlabel {\scriptsize{essential left arc}} at  15 135	
	\pinlabel {\scriptsize{initial right arc}} at  97 150	
	\pinlabel {\scriptsize{essential right arc}} at  110 60	
\endlabellist
\includegraphics{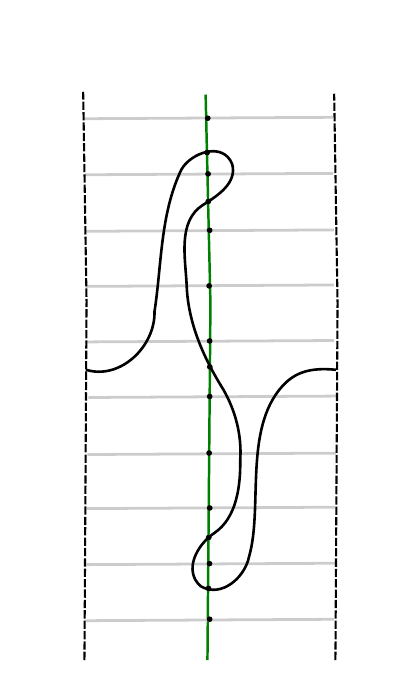}
\label{subfig:immersedcurve2}
}\end{tabular}
\caption{Examples of chain complexes over $\F[\U, \V]$ and their associated immersed curves.}
\label{fig:compextocurve}
\end{center}
\end{figure}

In the absence of a simultaneously horizontally and vertically simplified basis, one obtains an immersed curve together with a nontrivial local system. (Note that there exist knot-like complexes that do not admit a simultaneously horizontally and vertically simplified basis; see, for example, \cite[Figure 3]{Hom-infinite-rank}.) As noted at the end of Section 1.1 of \cite{HRW:2018}, a local system can be expanded by replacing the associated curve with parallel copies together with crossover arrows, which can be confined to a prescribed part of the curve. Note that if the crossover arrows are confined to a right arc, then the corresponding basis is vertically simplified. 

We denote the component of the immersed curve containing the essential right arc (and hence also the essential left arc) $\gamma_0$. As noted on page 2 of \cite{HW-cables}, for an immersed curve associated to the complement of a knot in $S^3$, the component $\gamma_0$ always carries the trivial 1-dimensional local system (otherwise the rank of $\HFhat$ of the meridional filling of $K$ would be greater than one; compare with items \eqref{it:knot-like1} and \eqref{it:knot-like2} of Definition \ref{def:knot-like}). We call the right arc of $\gamma_0$ containing $y_0$ the \emph{initial right arc}. Note that the initial right arc abuts the essential left arc. See Figure \ref{subfig:immersedcurve2}.

Conversely, an immersed curve (possibly with a nontrivial local system) gives a chain complex over $\F[\U, \V]/(\U \V)$, essentially by reversing the above procedure; see \cite[Theorem 2]{KWZ}. If the immersed curve is the bordered invariant associated to the complement of a knot in $S^3$, then the associated chain complex is the mod $(\U \V)$ reduction of a knot-like complex. Each intersection point of the immersed curve with the vertical axis yields a basis element $z_i$, and we have a horizontal (resp. vertical) arrow of length $n$ from $z_i$ to $z_j$ exactly when there is an upwards right arc (resp. downwards left arc) of length $n$ from $z_i$ to $z_j$. Here, we do not give the immersed curve an orientation; we only look locally at an arc and consider right (resp. left) arcs as upwards (resp. downwards).

\begin{remark}\label{rem:gamma0}
Note that under this correspondence between immersed curves and chain complexes, if we restrict to the component $\gamma_0$ (which always carries the trivial 1-dimensional local system), we may assume that our basis is simultaneously vertically and horizontally simplified.
\end{remark}

Suppose one is given a right arc which is either noninitial or initial such that $y_0$ is not contained in $\{0\} \times \left[-\frac{1}{2}, \frac{1}{2}\right]$. Observe that given a right arc of length $n$, there are two options for the behavior at each end of the arc: it may turn up or down. These four possibilities are illustrated in Figure \ref{fig:immersedcurve4cases}, and denoted $\eta_n^{\pm \pm}$, where the first superscript describes the behavior of the top end, the second superscript describes the behavior of the bottom end, and the subscript denotes the length of the arc. (We follow the notation of \cite[proof of Proposition 5]{HW-cables}.) An analogous statement holds for the ends of a left arc.  

\begin{figure}[htb!]
\begin{center}
\labellist
	\pinlabel {$\eta_n^{--}$} at  50 0
	\pinlabel {$\eta_n^{-+}$} at  155 0
	\pinlabel {$\eta_n^{+-}$} at  265 0
	\pinlabel {$\eta_n^{++}$} at  370 0
\endlabellist
\includegraphics{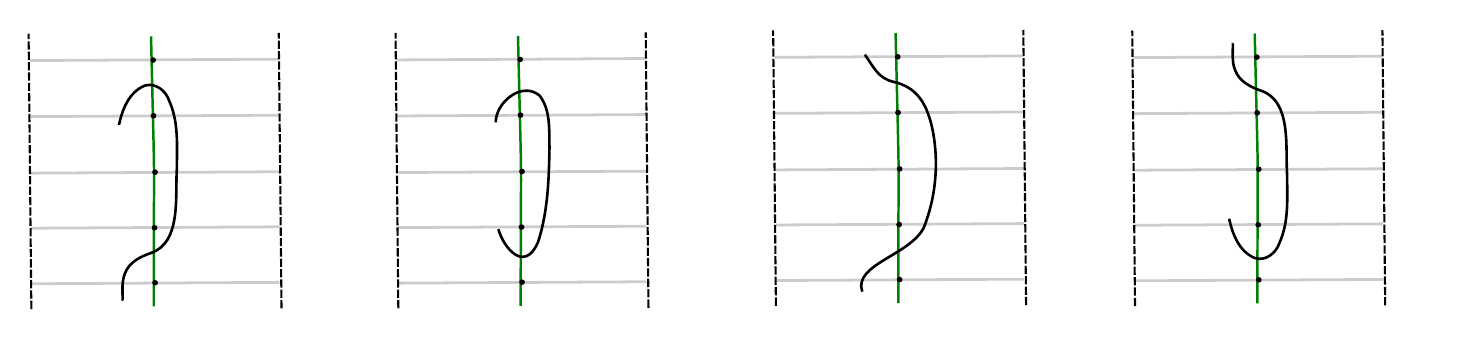}
\caption{The four different types of end behavior of a right arc when the right arc is either noninitial or initial such that $y_0$ is not contained in $\{0\} \times \left[-\frac{1}{2}, \frac{1}{2}\right]$. Here, $n=3$.}
\label{fig:immersedcurve4cases}
\end{center}
\end{figure}
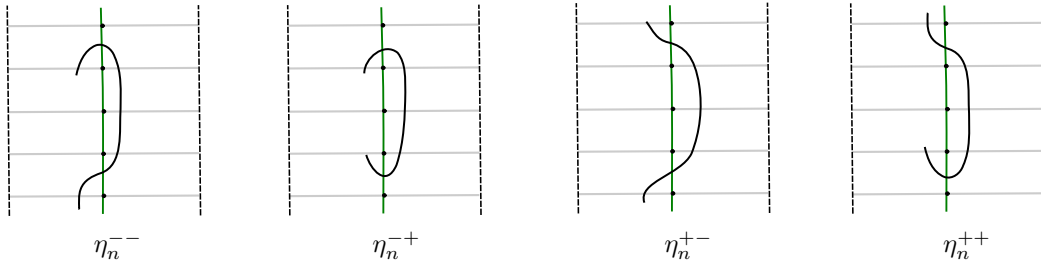

%of the essential left arc, the second superscript describes the behavior of the portion of the initial right arc that meets the essential left arc, and the last superscript describes the behavior of the other portion of the initial right arc.

For the initial right arc such that $y_0$ is contained in $\{0\} \times \left[-\frac{1}{2}, \frac{1}{2}\right]$, we have five more cases, denoted by $\eta^{0}$, $\eta_n^{\pm 0}$, and $\eta_n^{ 0 \pm}$ as in Figure~\ref{fig:immersedcurve5cases}. (Note that $\eta^0$ is a degenerate initial right arc.) As before the first superscript describes the behavior of the top end, the second superscript describes the behavior of the bottom end, and the subscript denotes the length of the arc.
\begin{figure}[htb!]
\begin{center}
\labellist
	\pinlabel {$\eta^{0}$} at  47 0
	\pinlabel {$\eta_n^{-0}$} at  155 0
	\pinlabel {$\eta_n^{+0}$} at  265 0
	\pinlabel {$\eta_n^{0+}$} at  370 0
	\pinlabel {$\eta_n^{0-}$} at  475 0
\endlabellist
\includegraphics[width=165mm]{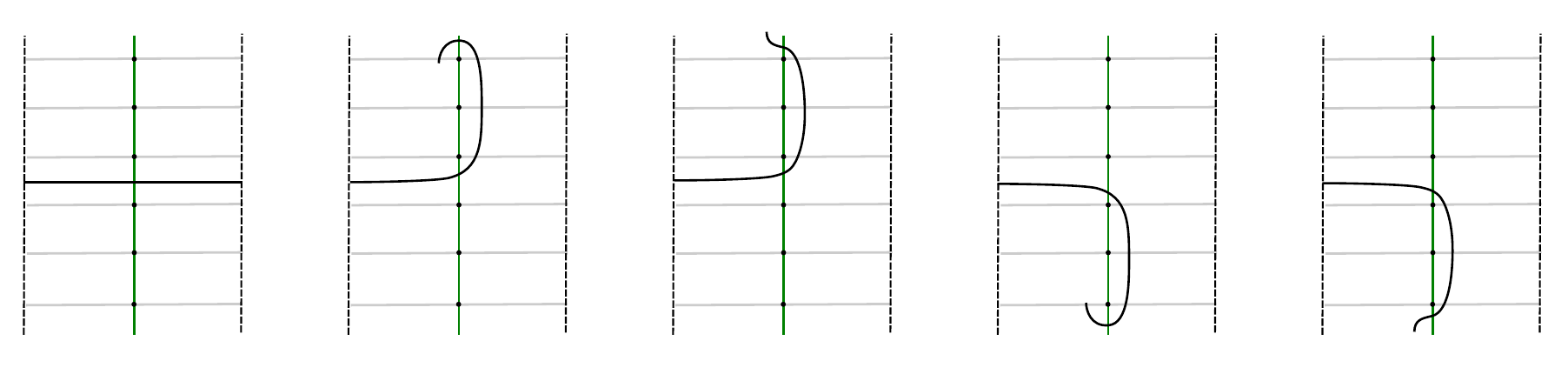}
\caption{The five different types of behavior of the initial right arc such that $y_0$ is contained in $\{0\} \times \left[-\frac{1}{2}, \frac{1}{2}\right]$. Here, $n=3$.}
\label{fig:immersedcurve5cases}
\end{center}
\end{figure}
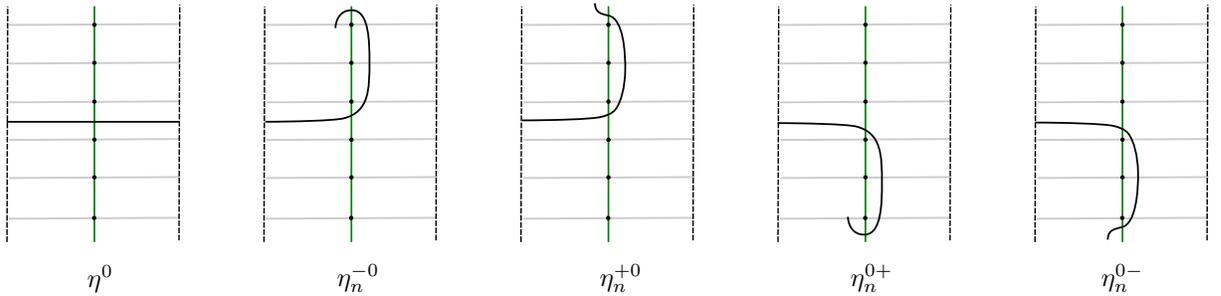

The following lemma is a consequence of the relationship between a knot-like complex and its associated immersed curve. From now on, we assume that the immersed curves are pulled tight so that for each arc there is at least one peg between two intersection points between the arc and $\{0\} \times \mathbb{R}$. (Note that this is the immersed curve analogue of assuming that a chain complex is reduced.)

\begin{lemma}\label{lem:rightarc}
Let $C$ be a knot-like complex and suppose that $H^-(C) = \F[\U] \oplus \bigoplus_{i=1}^N \F[\U]/(\U^{n_i})$. Then for each $i$, the immersed curve associated to $C$ has a right arc of length $n_i$. Conversely, an immersed curve with a right arc of length $n$ comes from a knot-like complex $C$ with an $\F[\U]/(\U^n)$ summand in $H^-(C)$.
\end{lemma}
\begin{proof}
Consider a horizontally simplified basis $\{x_i\}$ for $C$. Then it follows from the above discussion that each pair $\{x_i, x_{i+1}\}$, $i$ odd, where $\d x_i = \U^{n_i} x_{i+1} \mod (\V)$ gives rise to a right arc of length $n_i$. The converse statement follows similarly.
\end{proof}

\subsection{Concordance invariants and immersed curves}\label{sec:tauandepsilon}
The concordance invariants $\tau$ \cite{Ozsvath-Szabo:2003-1} and $\varepsilon$ \cite{Hom:2014-1} have simple interpretations in the immersed curve setting. Consider the component $\gamma_0$ associated to the complement of a knot $K$. Begin at $(-\frac{1}{2}, 0)$ and follow $\gamma_0$ to the right to $y_0$. Then $\tau(K)$ is the floor of the sum of $\frac{1}{2}$ and the vertical coordinate of $y_0$ (i.e.,\ the sum of $\frac{1}{2}$ and the vertical coordinate of the peg right below $y_0$.) If $\gamma_0$ turns downward at $y_0$, then $\varepsilon(K) = 1$. If $\gamma_0$ turns upwards at $y_0$, then $\varepsilon(K) = -1$. If $\gamma_0$ consists of a horizontal line, then $\varepsilon(K) = 0$. For example, in Figure \ref{subfig:immersedcurve2}, we have that $\tau(K) = 4$ and $\varepsilon(K) = 1$. In fact, using \cite[Theorem 1]{Homsurvey} Hanselman-Watson proved that $\gamma_0$ is a concordance invariant~\cite[Proposition 2]{HW-cables}.

\subsection{An obstruction to certain immersed curves}

The following proposition states that if $\{y_i\}$ is a vertically simplified basis for a knot-like complex and $i>0$ is even and $j$ is odd, then there cannot be a horizontal arrow of length one from $y_i$ to $y_j$. The proof relies on the fact that $\d^2 = 0$. See Figure \ref{fig:d2neq0}.

\begin{proposition}\label{prop:d2neq0}
Let $\{y_i\}$ be a vertically simplified basis for a knot-like complex. Then for all even $i > 0$ and odd $j$, we have $\langle \d y_i, y_j \rangle = 0 \mod (\U^2, \V)$.
\end{proposition}

\begin{proof}
Suppose for the sake of contradiction that $\langle \d y_i, y_j \rangle \neq 0 \mod (\U^2, \V)$ for some even $i > 0$ and odd $j$. Since our basis is assumed to be reduced and since our chain complexes are graded, it follows that $\langle \d y_i, y_j \rangle = \U$.

We claim that $\langle \d^2 y_{i-1}, y_{j} \rangle \neq 0$, contradicting that $\d^2 = 0$. Indeed, 
\[ \d y_{i-1} = \V^{m_{i-1}} y_i +\U \sum_{k \in K} \U^{a_k} \V^{b_k} y_{\ell_k} \]
for some  nonnegative integers $a_k, b_k$ and index set $K$. It follows that 
\begin{align*}
	\langle \d^2 y_{i-1}, y_{j} \rangle &= \langle \V^{m_{i-1}} \d y_i +\U \sum_{k \in K} \U^{a_k} \V^{b_k} \d y_{\ell_k}, y_j \rangle \\
		&= \langle \V^{m_{i-1}} \d y_i, y_j \rangle + \langle \U \sum_{k \in K} \U^{a_k} \V^{b_k} \d y_{\ell_k} , y_j \rangle.
\end{align*}
The first term above on the right hand side is $\V^{m_{i-1}} \U$, but by Remark \ref{rk:dxixj}, the second term is $0 \mod (\U^2)$, completing the proof.
\end{proof}

\begin{figure}[htb!]
\begin{center}
\begin{tikzpicture}[scale=1]
	\filldraw (3.5, 3.5) circle (2pt) node[label=right:{\lab{y_{i-1}}}] (b) {};
	\filldraw (3.5, 0.5) circle (2pt) node[label=right :{\lab{\V^{m_{i-1}} y_i}}] (c) {};
	\filldraw (2.5, 0.5) circle (2pt) node[label= left:{\lab{\U \V^{m_{i-1}}   y_j}}] (d) {};
	\filldraw (2.5, -1.5) circle (2pt) node[label= left:{\lab{\U \V^{m_j+m_{i-1}}  y_{j+1}}}] (f) {};

	\draw[->] (b) to (c);
	\draw[->] (d) to (f);
	\draw[->] (c) to (d);

\end{tikzpicture}
\caption{A graphical depiction of the proof of Proposition \ref{prop:d2neq0}. Since the basis is vertically simplified, there is no other term in $\d^2 y_{i-1}$ that can cancel with $\U \V^{m_{i-1}} y_j$. It is key that the horizontal arrow between $y_i$ and $y_j$ is length one; if it were longer, there could be a sequence of two arrows (each with a nonzero horizontal component) from $y_{i-1}$ to $y_j$ making $\d^2 y_{i-1} = 0$.}
\label{fig:d2neq0}
\end{center}
\end{figure}
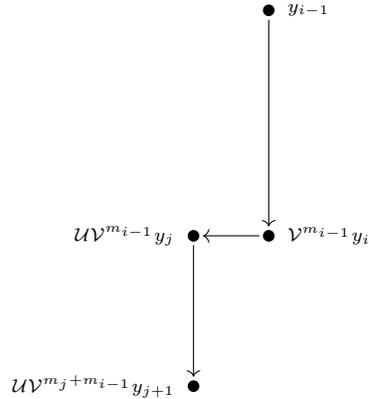

\begin{corollary}\label{cor:eta1}
With notation as in Figure \ref{fig:immersedcurve4cases}, a noninitial right arc of the form $\eta_1^{-+}$ cannot occur.
\end{corollary}

\begin{proof}
The result follows immediately from Proposition \ref{prop:d2neq0} and the discussion in Section \ref{sec:complextocurve}, translating between a chain complex over $\F[\U,\V]/(\U \V)$ and an immersed curve. Namely, a noninitial  right arc $\eta_1^{-+}$ corresponds to a vertically simplified basis $\{y_i\}$ with a horizontal arrow of length one from $y_i$ to $y_j$, $i > 0$ even and $j$ odd, which by Proposition \ref{prop:d2neq0} cannot occur.  
\end{proof}

%\footnote{\JP{why is $i>0$? I think the corollary is fine if we assume that this right arc is not an initial right arc. If it is an initial right arc then there are two cases either $\varepsilon$ is positive or negative. I think we don't care if it is negative and if it is positive we should use the next lemma.}}

We have an analogous result when $\varepsilon = 1$ and the initial right arc of the form $\eta_1^{\pm+}$ and $\eta_1^{0+}$.

\begin{lemma}\label{lem:eta1}
With notation as in Figure \ref{fig:immersedcurve4cases} and Figure~\ref{fig:immersedcurve5cases}, if $\varepsilon = 1$, then an initial right arc of the form $\eta_1^{\pm+}$ or $\eta_1^{0+}$ cannot occur. \end{lemma}

\begin{proof}
The result follows from Lemma 3.7 of \cite{Hom-infinite-rank}, which, in the language of immersed curves, states that if $\varepsilon = 1$ and $\eta_n^{\pm+}$ or $\eta_n^{0+}$ is the initial right arc, then $n \geq 2$.  (We remind the reader that the sign conventions of the $a_i$-invariants in \cite{Hom-infinite-rank} differ from the conventions in \cite{DHSTmore}.)  Since the initial right arc is a part of $\gamma_0$, by Remark~\ref{rem:gamma0} we may assume we are working with a horizontally and vertically simplified basis.

For the convenience of the reader, we sketch the proof. If $\varepsilon = 1$, then an initial right arc of the form $\eta_n^{\pm+}$ or $\eta_n^{0+}$ translates to a vertically simplified basis $\{y_i\}$ with a horizontal arrow of length $n$ from $y_{i+1}$ to $y_0$ and a vertical arrow of length $m_i$ from $y_i$ to $y_{i+1}$, $i$ odd. If $n=1$, then since $\{y_i\}$ is vertically simplified, we have $\langle \d^2 y_{i}, y_{0} \rangle \neq 0 \mod (\U\V)$, contradicting that $\d^2 = 0$. Hence, the result follows.  
\end{proof}

\section{Cables}\label{sec:cables}
In this section, we use \cite[Theorem 1]{HW-cables} to study the effect of cabling on $\HFKm(K)$ by understanding the effect of cabling on right arcs of the immersed curve associated to $K$, from which Theorems \ref{thm:main} and \ref{thm:iteratedtorus} follow readily. Note that the following proposition only describes a subset of the right arcs produced by $(p,q)$-cabling.

\begin{proposition}\label{prop:cabling}
Consider the immersed curve associated to a knot-like complex.  
\begin{enumerate}[font=\upshape]

\item \label{it:notessential} Consider a noninitial right arc of the form $\eta_n^{\pm \pm}$. Under the operation of $(p, q)$-cabling:

\begin{enumerate}[font=\upshape]
\item \label{it:notessential1a} A right arc $\eta_n^{--}$ yields a right arc of length $pn$.
\item \label{it:notessential1b} A right arc $\eta_n^{-+}$ yields a right arc of length $pn-p+1$.
\item \label{it:notessential1c} A right arc $\eta_n^{+-}$ yields a right arc of length $pn+p-1$.		
\item \label{it:notessential1d} A right arc $\eta_n^{++}$ yields a right arc of length $pn$.
\end{enumerate}
Moreover, in this case the resulting arc is noninitial and has the same form as the original right arc.

\item \label{it:essential2} Suppose $\varepsilon = 1$ and consider an initial right arc of the form $\eta_n^{\pm \pm}$ or $\eta_n^{0\pm}$. Under the operation of $(p, q)$-cabling:

\begin{enumerate}[font=\upshape]
\item\label{it:essential2a} 
Right arcs $\eta_n^{--}$, $\eta_n^{+-}$, and $\eta_n^{0-}$ yield right arcs of length

\begin{tabular}{@{\quad $\boldsymbol{\cdot}$ }lll}
\hspace{.05cm}$pn$      & \quad  for \quad $q<p(2\tau -1)$,\\
\hspace{.05cm}$pn+p-2p\tau+q-1$ &\quad  for \quad  $p(2\tau -1)<q<2p\tau$,\\
\hspace{.05cm}$pn+p-1$   & \quad  for \quad  $2p\tau<q$.
\end{tabular}
%
%
%$$\begin{cases}
%pn & \text{\quad  for \quad $q<p(2\tau -1)$,}\\
%pn+p-2p\tau+q-1 & \text{\quad  for \quad  $p(2\tau -1)<q<2p\tau$,}\\
%pn+p-1 & \text{\quad  for \quad  $2p\tau<q$.}
%\end{cases}$$

\item\label{it:essential2b} 
Right arcs $\eta_n^{-+}$, $\eta_n^{++}$, and $\eta_n^{0+}$ yield right arcs of length

\begin{tabular}{@{\quad $\boldsymbol{\cdot}$ }lll}
\hspace{.05cm}$pn-p+1$      & \quad  for \quad $q<p(2\tau -1)$,\\
\hspace{.05cm}$pn-2p\tau+q $ \phantom{$+p-1$} &\quad  for \quad  $p(2\tau -1)<q<2p\tau$,\\
\hspace{.05cm}$pn$   & \quad  for \quad  $2p\tau<q$.
\end{tabular}
\end{enumerate}
Furthermore, the right arc obtained in each case from \eqref{it:essential2a} and \eqref{it:essential2b} is noninitial and the second superscript does not change.
\end{enumerate}
\end{proposition}

Before we prove Proposition \ref{prop:cabling}, we first show how it implies Propositions \ref{prop:main} and \ref{prop:iteratedtorus}, whose statements we recall.

\begin{proposition}\label{prop:mainbody}
If $K$ is a nontrivial knot, then 
$$\Ord(K_{p_1, q_1; p_2, q_2; \dots ; p_m, q_m}) \geq \max \left\{ p_1  p_2 \dots  p_m \left(\Ord(K)-1\right)+1, p_1  p_2 \dots  p_m \right\}.$$
\end{proposition}

\begin{proof}
Let $K$ be a nontrivial knot. Since $\Ord(K) = \Ord(-K)$~\cite[Section 5.1]{Ozsvath-Szabo:2006-1} and $\varepsilon(K) = -\varepsilon(-K)$~\cite[Proposition 3.6]{Hom:2014-1b}, by mirroring if necessary, we may assume that $\varepsilon(K) \geq 0$. As observed in \cite{AlishahiEftekharyunknotting}, since knot Floer homology detects the unknot \cite[Theorem 1.2]{OSgenus}, it follows that $\Ord(K) \geq 1$. Writing $\Ord(K) =n$, Lemma \ref{lem:rightarc} implies that the immersed curve associated to $K$ has a right arc of the form $\eta_n^{\pm \pm}$ or $\eta_n^{0\pm}$ with $n \geq 1$.  We begin with the case of $m = 1$.  

%\textcolor{red}{TL: I assume this notation means $\eta_n^{+ -}$ is ok?  Do we have places where we want $\pm \pm$ to really mean only $+ +$ or $- -$? JP: Yes, we mean all four cases. Right, we don't have a place where we mean only  $+ +$ or $- -$.}

We first consider the case where we have a noninitial right arc of the form $\eta_n^{\pm \pm}$ with $n \geq 1$. We consider the four cases in part \eqref{it:notessential} of Proposition \ref{prop:cabling} under the operation of $(p_1,q_1)$-cabling:
\begin{enumerate}[(a)] 
	\item\label{a} A right arc of the form $\eta_n^{--}$ yields a right arc of length $p_1 n$. Since $n \geq 1$, we have that $p_1n \geq \max \left\{ p_1(n-1)+1, p_1  \right\}$.
	\item\label{b} A right arc of the form $\eta_n^{-+}$ yields a right arc of length $p_1(n-1)+1$. Corollary \ref{cor:eta1} states that $\eta_1^{-+}$ cannot occur, hence $n \geq 2$, which implies that $p_1(n-1)+1\geq \max \left\{p_1(n-1)+1, p_1  \right\}$.
	\item\label{c} A right arc of the form $\eta_n^{+-}$ yields a right arc of length $p_1(n+1) -1$.  Since $n \geq 1$, we have that  $p_1(n +1)-1 \geq \max \left\{p_1(n-1)+1, p_1  \right\}$.
	\item\label{d} A right arc of the form $\eta_n^{++}$ yields a right arc of length $p_1 n$. Since $n \geq 1$, we have that $p_1 n \geq \max \left\{p_1(n-1)+1, p_1  \right\}$.
\end{enumerate}
We see that in all cases, we have a right arc of length at least $\max \left\{p_1(n-1)+1, p_1  \right\}$. 

Now suppose that $\eta_n^{\pm \pm}$ is the initial right arc, which implies that $\varepsilon(K) \neq 0$; hence $\varepsilon(K) = 1$. The desired result now follows from part \eqref{it:essential2} of Proposition \ref{prop:cabling}. That is, we obtain a right arc of length at least $\max \left\{p_1(n-1)+1, p_1  \right\}$ in all cases. This is clear for all cases except \eqref{it:essential2b} when $q<p(2\tau-1)$. Lemma~\ref{lem:eta1} states that $\eta_1^{-+}$ cannot occur; hence in \eqref{it:essential2b}, $n \geq 2$. Thus, in all cases, we have a right arc of length at least $\max \left\{p_1(n-1)+1, p_1  \right\}$.  This completes the proof for $m = 1$.

%Similarly, Lemma \ref{lem:eta1} states that an initial right arc of the form $\eta_1^{++}$ cannot occur; hence in \eqref{it:essential3d}, $n \geq 2$.

%\textcolor{red}{TL: Do we need an argument here for the case of $\eta_n^{0 \pm }$ and again something about Case 5a?  The rest of the proof seems to be about dealing with $m \geq 2$ but I don't see the $m = 1$ case for Case 5 anywhere. JP: Aha, thanks for catching that. I added this case below.}  
%
%The case where we have an initial right arc of the form $\eta_n^{0 \pm}$ with $n \geq 1$ follows from parts \eqref{it:essential4} and \eqref{it:essential5} of Proposition \ref{prop:cabling}. More precisely, the proposition implies that we obtain a right arc of length at least $\max \left\{p_1(n-1)+1, p_1  \right\}$ in all cases. As before, this is clear for all cases except \eqref{it:essential5a}, but Lemma~\ref{lem:eta1} implies that $n \geq 2$ for this case. Hence, we have a right arc of length at least $\max \left\{p_1(n-1)+1, p_1  \right\}$ in all cases. This completes the proof for $m = 1$.

For $m > 1$, we consider two cases separately. If $n>1$, then the resulting new arc for $K_{p_1,q_1}$ is noninitial by Proposition~\ref{prop:cabling} and has length at least $p_1(n-1)+1=\max \left\{p_1(n-1)+1, p_1  \right\}$. Repeating the analysis from \ref{a}, \ref{b}, \ref{c}, and \ref{d}, we see that $$\Ord(K_{p_1, q_1; p_2, q_2; \dots ; p_m, q_m}) \geq p_1  p_2 \dots  p_m \left(n-1\right)+1.$$ If $n=1$, as we have seen from above that the arc that we have started with cannot be a noninitial arc with the form $\eta_1^{-+}$, an initial arc of the form $\eta_1^{\pm +}$, or an initial arc of the form $\eta_1^{0 +}$. Hence the resulting new arc for $K_{p_1,q_1}$ is noninitial with the form $\eta_n^{--}$, $\eta_n^{+-}$, or $\eta_n^{++}$ by Proposition~\ref{prop:cabling} and has length at least $p_1 =\max \left\{p_1(n-1)+1, p_1  \right\}$. Then repeating the analysis from \ref{a}, \ref{c}, and \ref{d}, we see that $$\Ord(K_{p_1, q_1; p_2, q_2; \dots ; p_m, q_m}) \geq p_1  p_2 \dots  p_m.$$ This concludes the proof.
\end{proof}

\begin{proposition}\label{prop:iteratedtorusmainbody}
Let $K$ be an L-space knot and $K\neq T_{2,3}$. If $J= K_{p_1, q_1; p_2, q_2; \dots ; p_m, q_m}$, then $\gamma_0$ associated to the complement of $J$ contains a right arc with length at least $p_1 \dots p_m \left(\Ord(K)+1\right)-1$. In particular,
 $$\Ord(J) \geq p_1 \dots p_m \left(\Ord(K)+1\right)-1.$$
Moreover, if the bridge index of $K$ is $\Ord(K)+1$, then the above inequality is an equality.
\end{proposition}

\begin{proof}
Suppose $K$ is an L-space knot and $K\neq T_{2,3}$. Recall from \cite[Theorem 1.2]{OSlens} (see also, for example, \cite[Section 7]{HendricksManolescu}) that for an L-space knot, the knot Floer complex is determined by the Alexander polynomial. 
%\textcolor{red}{TL: Should we put in something about staircases here, since we need that to get the right arc?}\JP{I made some changes. Let me know what you think.}  \textcolor{red}{TL: Do we need the full staircase structure of $CFK^\infty$ to guarantee the $\eta^{+-}_n$ or am I missing something that it follows more directly from $\Delta_K$? JP: Aha you are right. I added a few sentences. Let me know what you think.}  
Moreover, they showed that the Alexander polynomial of $K$ has the following form:
\[ \Delta_{K} (t) = \sum_{k=0}^{2\ell} (-1)^k t^{\alpha_k}\]
where $\alpha_0, \ldots , \alpha_{2\ell}$ is a decreasing sequence of integers and $\alpha_0 - \alpha_{2\ell}$ is twice the genus of $K$~\cite[Corollary 1.3]{OSlens}. Moreover, since $T_{2,3}$ is the unique genus one L-space knot~\cite{Ni:2007-1, Burde:1967-1, Gonzalez-Acuna:1970-1} and $K\neq T_{2,3}$, we have that $\alpha_0 - \alpha_{2\ell}>2$. Lastly, Hedden-Watson~\cite[Corollary 9]{Hedden-Watson:2018-1} showed that $\alpha_0 = \alpha_1 +1$ and $\alpha_{2\ell-1} = \alpha_{2\ell} +1$. In particular, there is an $\alpha_i$ such that $\alpha_1 > \alpha_i > \alpha_{2\ell -1}$.  

%It follows that the immersed curve $\gamma_0$ associated to $K$ contains a noninitial right arc of length $n$ where $n=\Ord(K)$. 

The torsion order of an L-space knot may be computed easily from its Alexander polynomial~\cite[Lemma 5.1]{Juhasz-Miller-Zemke:2020-1}. More precisely, we have
\begin{equation}\label{eq:orderlspace}
\Ord(K) = \max \left\{\alpha_{i-1}-\alpha_i \mid 1\leq i \leq 2\ell \right\}.\end{equation}
Moreover, it is well known that the knot Floer complex of an L-space knot has a special form. More precisely, the differentials in the complex form a staircase (see e.g.\ \cite[Section 7]{HendricksManolescu}). Hence, following the discussion in Section~\ref{sec:complextocurve}, it can be easily verified that in fact there is a noninitial right arc of the form $\eta_{n}^{+-}$ in $\gamma_0$ associated to $K$ where $n=\Ord(K)$. Part \eqref{it:notessential1c} of Proposition \ref{prop:cabling} implies that the immersed curve $\gamma_0$ associated to $K_{p_1, q_1}$ has a right arc of the form $\eta^{+-}$ of length $p_1n+p_1-1 = p_1(n+1)-1$. The desired inequality now follows from repeatedly applying part \eqref{it:notessential1c} of Proposition \ref{prop:cabling}.

%The torsion order of an L-space knot may be computed easily from its Alexander polynomial~\cite[Lemma 5.1]{Juhasz-Miller-Zemke:2020-1}. More precisely, we have
%\begin{equation}\label{eq:orderlspace}
%\Ord(K) = \max \left\{\alpha_{i-1}-\alpha_i \mid 1\leq i \leq 2\ell \right\}.\end{equation}
%It follows that the immersed curve $\gamma_0$ associated to $K$ contains a noninitial right arc of the form $\eta_{n}^{+-}$ where $n=\Ord(K)$. Part \eqref{it:notessential1c} of Proposition \ref{prop:cabling} implies that the immersed curve $\gamma_0$ associated to $K_{p_1, q_1}$ has a right arc of the form $\eta^{+-}$ of length $p_1n+p_1-1 = p_1(n+1)-1$. The desired inequality now follows from repeatedly applying part \eqref{it:notessential1c} of Proposition \ref{prop:cabling}. 

For the equality, it is a classical theorem of Schubert~\cite{Schubert:1954-1} that the bridge index of $J$ is equal to $p_1p_2\cdots p_m$ times the bridge index of $K$. Moreover, in \cite[Corollary 1.8]{Juhasz-Miller-Zemke:2020-1} they prove that the torsion order of a knot is less than the bridge index of the knot. Hence, we conclude that if the bridge index of $K$ is $\Ord(K)+1$, then the torsion order of $J$ is $p_1 \dots p_m \left(\Ord(K)+1\right)-1$.
\end{proof}

%
%\begin{proof}[Proof of Theorem \ref{thm:iteratedtorus}]
%Suppose $p_0 < q_0$ and $T_{p_0, q_0} \neq T_{2,3}$. Recall from \cite[Theorem 1.2]{OSlens} (see also, for example, \cite[Section 7]{HendricksManolescu}) that for an L-space knot, the knot Floer complex is determined by the Alexander polynomial. \textcolor{red}{TL: Should we put in something about staircases here, since we need that to get the right arc? JP: That sounds great!}  Since $T_{p_0, q_0}$ is an L-space knot and $\Delta_{T_{p_0,q_0}} (t) = \frac{(t^{p_0q_0}-1)(t-1)}{(t^{p_0}-1)(t^{q_0}-1)}$, a straightforward calculation implies that 
%\[ \Delta_{T_{p_0,q_0}} (t) = t^{2g} - t^{2g-1} + t^{2g-p_0} - \dots + t^{p_0} - t + 1\]
%where $g=\frac{(p_0-1)(q_0-1)}{2}$. Note that $2g-p_0 >0$ since we are assuming $T_{p_0, q_0} \neq T_{2,3}$. It follows that the immersed curve associated to $T_{p_0, q_0}$ contains a noninitial right arc of the form $\eta_{p_0-1}^{+-}$. Part \eqref{it:notessential1c} of Proposition \ref{prop:cabling} implies that the immersed curve associated to $T_{p_0, q_0; p_1, q_1}$ has a right arc of the form $\eta^{+-}$ of length $p_1(p_0-1)+p_1-1 = p_1 p_0 -1$. The result now follows from repeatedly applying part \eqref{it:notessential1c} of Proposition \ref{prop:cabling}, and the Alishahi-Eftekhary unknotting bound in \eqref{eq:unknotting}.
%\end{proof}

%\purple{We make a remark that a similar lower bound can be obtained if we replace $T_{p_0, q_0}$ with either an $L$-space knot or the mirror image of an $L$-space knot.}

We now prove Proposition \ref{prop:cabling}.

\begin{proof}[Proof of Proposition \ref{prop:cabling}]
Recall that \cite[Theorem 1]{HW-cables} states that the immersed curve $\gamma_{p,q}$ corresponding to the $(p,q)$-cable of $K$ can be obtained from the immersed curve $\gamma$ for $K$ by the following three step procedure: 
\begin{enumerate}[(i)]
	\item \label{it:step1} Draw $p$ copies of $\gamma$ from left to right, each scaled vertically by a factor of $p$, staggered in height such that each copy is $q$ units below the preceding copy to the left.  
	\item \label{it:step2} Connect the loose ends of successive copies of the curve. 	
	\item \label{it:step3} Translate the pegs horizontally so that they lie in the same vertical line, carrying the curves along with them.  Pull the curves tight so that there is at least one peg between two intersection points of each arc and $\{0\} \times \mathbb{R}$. 
	Observe that the left end is $(p-1)q$ units higher than the right end. We shift the left end $\frac{(p-1)q}{2}$ units down and the right end $\frac{(p-1)q}{2}$ units up, so that the two ends points have the same vertical coordinate. We normalize the vertical coordinates so that these two end points are at $(-\frac{1}{2},0)$ and $(\frac{1}{2},0)$, respectively.
	%\textcolor{red}{Old: For the image of the essential left arc that has an end in $\{-\frac{1}{2}\} \times \mathbb{R}$ and the image of the essential right arc that has an end in $\{\frac{1}{2}\} \times \mathbb{R}$, drag these ends the equal length so that their vertical coordinates coincide. Then translate the whole pegs vertically so that the coordinates of these ends are $(-\frac{1}{2},0)$ and $(\frac{1}{2},0)$, respectively. }

%	Lastly, if there is an arc that connects a point in $\{0\} \times \mathbb{R}$ to a point in $\{\frac{1}{2}\} \times \mathbb{R}$, then drag the right end point from the arc to $(\frac{1}{2},0).$ 
	
\end{enumerate}

For a picture of Step~\ref{it:step1} see Figure~\ref{fig:immersedcurve4cases2} and for Step~\ref{it:step2}, see Figure~\ref{fig:initialsecond}.  We will analyze the contribution of the rightmost copy, namely the $p$th copy, of $\gamma$ from Step~\ref{it:step1} to the curve $\gamma_{p,q}$. We begin by considering the four different types $\eta_n^{\pm \pm}$ of right arc, and assume that $\eta_n^{\pm \pm}$ is a noninitial right arc.  While there may be a nontrivial local system present, this will not be relevant to the proof since this part of the argument only makes use of right arcs and we may confine the nontrivial local system to left arcs.

We first consider \eqref{it:notessential1a}, that is, a right arc $\eta_n^{--}$, scaled vertically by a factor of $p$. Since $\eta^{--}_n$ is noninitial, we see that after translating the pegs horizontally and pulling tight, the image of $\eta_n^{--}$ spans $pn$ pegs, as each of the $n$ pegs originally spanned by $\eta_n^{--}$ have now become $p$ pegs (consisting of the original peg and $p-1$ pegs below it), for a total of $pn$ pegs.
%Old: When we translate the pegs horizontally, we see that the image of $\eta_n^{--}$ spans $pn$ pegs, as each of the $n$ pegs originally spanned by $\eta_n^{--}$ have now become $p$ pegs (consisting of the original peg and $p-1$ pegs below it), for a total of $pn$ pegs. 

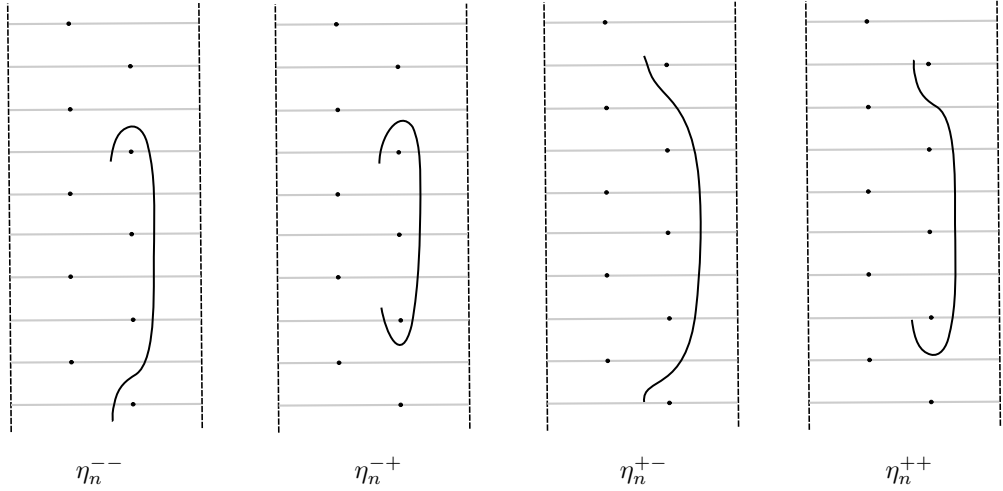
\begin{figure}[htb!]
\begin{center}
\labellist
	\pinlabel {$\eta_n^{--}$} at  50 10
	\pinlabel {$\eta_n^{-+}$} at  155 10
	\pinlabel {$\eta_n^{+-}$} at  255 10
	\pinlabel {$\eta_n^{++}$} at  355 10
\endlabellist
\includegraphics{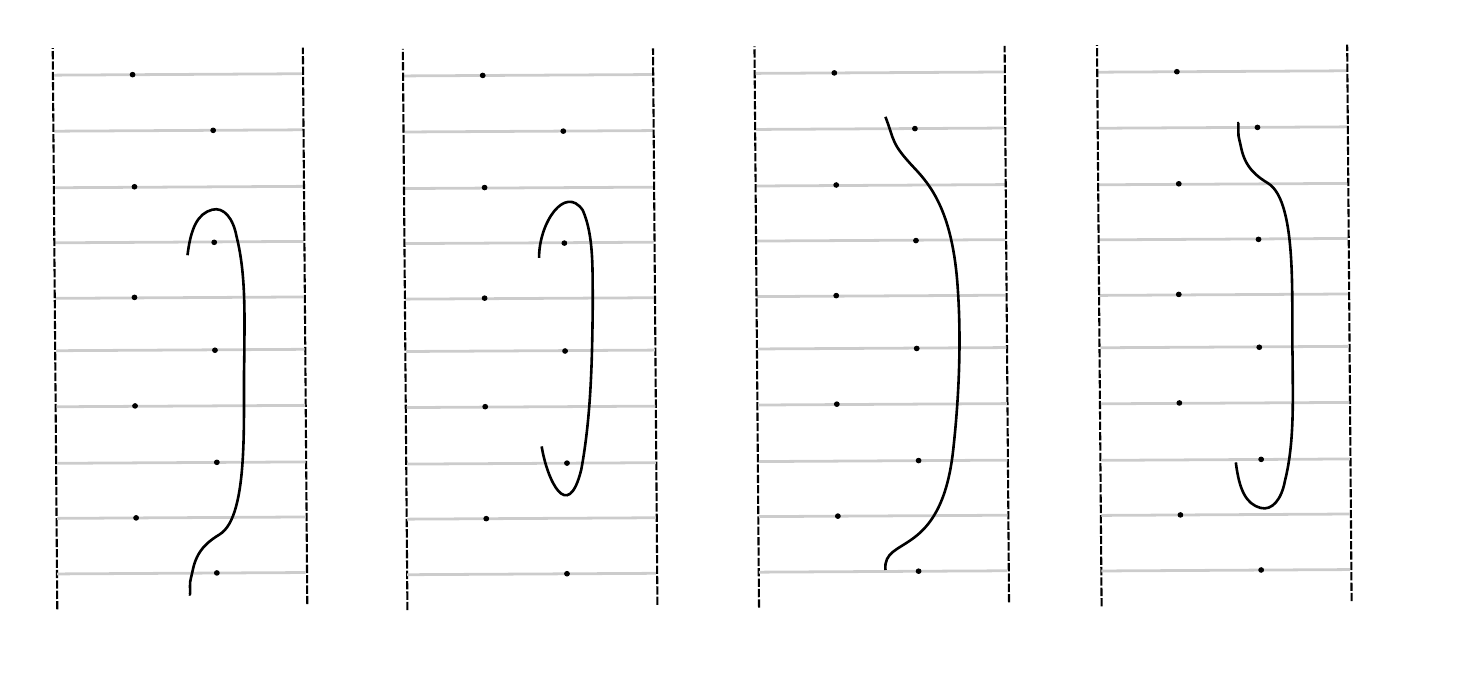}
\caption{The rightmost copy of a right arc of $\gamma$ from Step \ref{it:step1}, with the four different types of end behavior. Here, $n=3$ and $p=2$.}
\label{fig:immersedcurve4cases2}
\end{center}
\end{figure}

We next consider \eqref{it:notessential1b}, that is, a right arc $\eta_n^{-+}$, scaled vertically by a factor of $p$. When we translate the pegs horizontally, we see that the image of $\eta_n^{-+}$ spans $p(n-1)+1$ pegs, as each of the first $n-1$ pegs originally spanned by $\eta_n^{-+}$ have now become $p$ pegs (consisting of the original peg and $p-1$ pegs below it), plus we still have the last of the original $n$ pegs, for a total of $p(n-1)+1=pn-p+1$ pegs.

%When we translate the pegs horizontally, we see that the image of $\eta_n^{-+}$ spans $pn-p+1$ pegs, coming from the $n$ pegs spanned by the scaled copy of $\eta_n^{-+}$ together with $p-1$ below the first $n-1$ of these pegs, for a total of $n+(n-1)(p-1)+1 = pn -p+1$ pegs.

We next consider \eqref{it:notessential1c}, that is,  a right arc $\eta_n^{+-}$, scaled vertically by a factor of $p$. When we translate the pegs horizontally, we see that the image of $\eta_n^{+-}$ spans $pn+p-1$ pegs, as each of the original $n$ pegs has now become $p$ pegs (consisting of the original peg and the $p-1$ pegs below it), plus an additional $p-1$ pegs above the first of the original pegs. 

%Moreover, the resulting arc is of type $\eta^{+-}$.

Lastly, we consider \eqref{it:notessential1d}, a right arc of the form $\eta_n^{++}$; the result follows as in the case of $\eta_n^{--}$.

Furthermore, it is clear that if $\eta_n^{\pm \pm}$ is a noninitial right arc, then the right arc obtained in each of the above four cases is a noninitial right arc and the form does not change.  This completes the proof for noninitial right arcs, i.e. Part \eqref{it:notessential} of the proposition.

\begin{figure}[htb!]
\begin{center}
\labellist

\pinlabel {$u_0$} at  50 100
\pinlabel {$u'_0$} at  50 84
\pinlabel {$v_0$} at  50 52
\pinlabel {$w_0$} at  50 37

\endlabellist
\includegraphics[width=45mm]{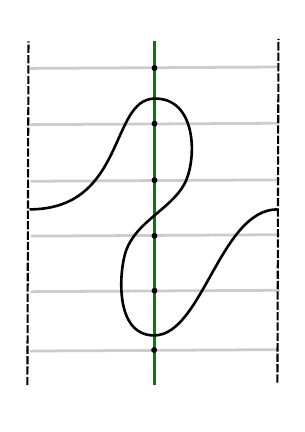}
\caption{Pegs $u_0$, $u'_0$, $v_0$, and $w_0$ of an initial right arc of the form $\eta_n^{--}$ with $\varepsilon =1$ and $\tau>0$.}
\label{fig:initialfirst}
\end{center}
\end{figure}
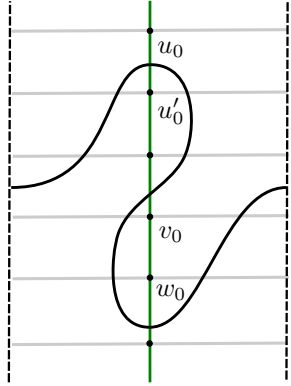

We next consider the initial right arcs $\eta_n^{\pm \pm}$ and  $\eta_n^{0 \pm}$ with $\varepsilon =1$.
% \textcolor{red}{TL: I can't tell where we are using $\epsilon = 1$ and where we aren't.  Should we just assume this from the start here? JP: sounds good. Added the assumption.} 
We only give a detailed proof for the case $\eta_n^{--}$ with $\varepsilon =1$.  The rest of the cases can be proved similarly. Note that given an initial right arc of the form $\eta_n^{--}$ and $\varepsilon = 1$, we have $\tau>0$.  Further, we will not need to consider local systems since we will now restrict to working with $\gamma_0$.

% Again, consider the rightmost copy of $\gamma$ from Step~\ref{it:step1} of Hanselman-Watson's three step procedure.  

Let $u_0$ be the peg at $\left(0,\tau + \frac{1}{2}\right)$, $u'_0$ be the peg at $\left(0,\tau - \frac{1}{2}\right)$, $v_0$ be the peg at $\left(0,\tau - \frac{1}{2} - n\right)$, and $w_0$ be the peg at $\left(0,-\tau +\frac{1}{2}\right)$ for $\gamma$. Note that $u_0$ and $u'_0$ are the pegs that are located right above and below the intersection point between the essential left arc and the vertical line $\{0\} \times \mathbb{R}$, respectively, and $w_0$ is the peg right above the intersection point between the essential right arc and the vertical line $\{0\} \times \mathbb{R}$ (see Section~\ref{sec:tauandepsilon} and Figure~\ref{fig:initialfirst}). Here, we are using the fact that there is a symmetry of the immersed curves under 180 degree rotation~\cite[Theorem~7]{HRW:2018}. Lastly, note that there are the two intersection points between the initial right arc and the vertical line $\{0\} \times \mathbb{R}$, and $v_0$ is the peg right below the bottom intersection point (see Figure~\ref{fig:initialfirst}).

%Let $u_0$ be the peg at $\left(0,\tau + \frac{1}{2}\right)$, $v_0$ be the peg at $\left(0,\tau - \frac{1}{2}\right)$, and let $w_0$ be the peg at $\left(0,-\tau +\frac{1}{2}\right)$ for $\gamma$. Note $u_0$ and $v_0$ are the pegs that are located right above and below the intersection between the essential left arc and the vertical line $\{0\} \times \mathbb{R}$, respectively, and $w_0$ is the peg right above the intersection between the essential right arc and the vertical line (see Section~\ref{sec:tauandepsilon}). Here, we are using the fact that there is a symmetry of the immersed curves under 180 degree rotation~\cite[Theorem~7]{HRW:2018}. 

%Now, consider the pegs, denoted by $u$ and $v$ respectively, for the immersed curve $\gamma_{p,q}$ which corresponds to $u_0$ and $v_0$, respectively, in the $p$th copy of $\gamma$ from Step \ref{it:step1} of \cite[Theorem 1]{HW-cables}. 

Let $u$ and $u'$ be the pegs for the immersed curve $\gamma_{p,q}$  in the $p$th copy of $\gamma$ from Step \ref{it:step1} of Hanselman-Watson's procedure that correspond to $u_0$ and $u'_0$, respectively. 
The peg that corresponds to $v_0$ in the $p$th copy of $\gamma$ is denoted by $v$ and the peg that corresponds to $w_0$ in the $(p-1)$th copy of $\gamma$  is denoted by $w$ (see Figure~\ref{fig:initialsecond}). By abusing notation, the vertical coordinates of the pegs $u$, $u'$, $v$, and $w$ will also be denoted by $u$, $u'$, $v$, and $w$, respectively. With this notation, note that we have $u'<u$.

%\begin{figure}[htb!]
%\begin{center}
%\subfigure[]{
%\labellist
%	\pinlabel {\lab{v}} at  46 119
%	\pinlabel {\lab{w}} at  54 94
%\endlabellist
%\includegraphics{figures/initialfirst}\label{subfig:initialfirst}
%}
%
%\caption{(a) An example of an immersed curve $\gamma$ with $\tau=1$ and pegs $v$ and $w$. (b) Pegs $v'$ and $w'$ after Step \ref{it:step1} and \ref{it:step2}. Here, $p=2$ and $q=3$.}
%\label{fig:initial}
%\end{center}
%\end{figure}

%
%\begin{figure}[htb!]
%\begin{center}
%\subfigure[]{
%\labellist
%	\pinlabel {\lab{v}} at  46 119
%	\pinlabel {\lab{w}} at  54 94
%\endlabellist
%\includegraphics{figures/trefoil}\label{subfig:trefoil}
%}
%\hspace{1cm}
%\subfigure[]{
%\labellist
%	\pinlabel {\lab{v'}} at  57 93
%	\pinlabel {\lab{w'}} at  44 119
%\endlabellist
%\includegraphics{figures/trefoilcabled}\label{subfig:trefoilcabled}
%}
%\caption{(a) An example of an immersed curve $\gamma$ with $\tau=1$ and pegs $v$ and $w$. (b) Pegs $v'$ and $w'$ after Step \ref{it:step1} and \ref{it:step2}. Here, $p=2$ and $q=3$.}
%\label{fig:initial}
%\end{center}
%\end{figure}

\begin{figure}[htb!]
\begin{center}
\labellist

\pinlabel {$u$} at  92 117
\pinlabel {$u'$} at  93.2 102.5
\pinlabel {$w$} at  33 79.3
\pinlabel {$v$} at  92 72
\pinlabel {$(1)$ $w< u'$} at  63 0

\pinlabel {$u$} at  232 102
\pinlabel {$u'$} at  233.2 87.8
\pinlabel {$w$} at  175 94.3
\pinlabel {$v$} at  232 57.5
\pinlabel {$(2)$ $u'<w<u $} at  202 0

\pinlabel {$u$} at  371.8 87
\pinlabel {$u'$} at  373 72.8
\pinlabel {$w$} at  317 108.9
\pinlabel {$v$} at  371.8 42.6
\pinlabel {$(3)$ $u<w$} at  345 0

\endlabellist
\includegraphics[width=165mm]{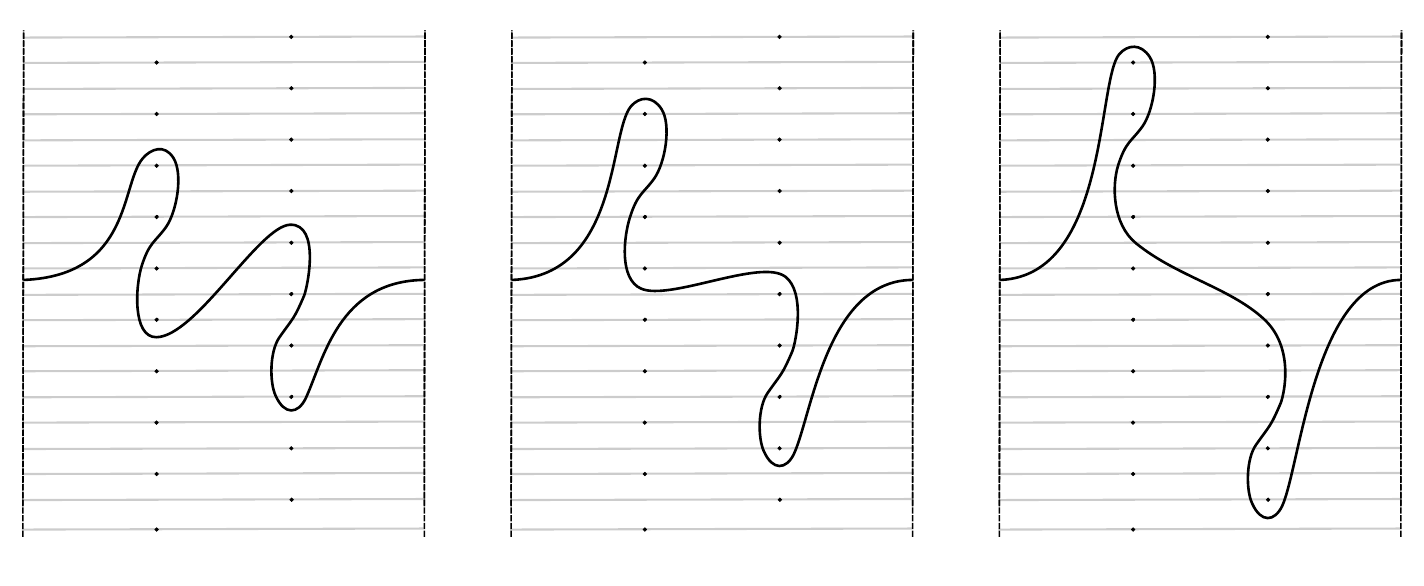}
\caption{Pegs $u$, $u'$, $v$, and $w$ for the immersed curve $\gamma_{p,q}$ for three cases.}
\label{fig:initialsecond}
\end{center}
\end{figure}
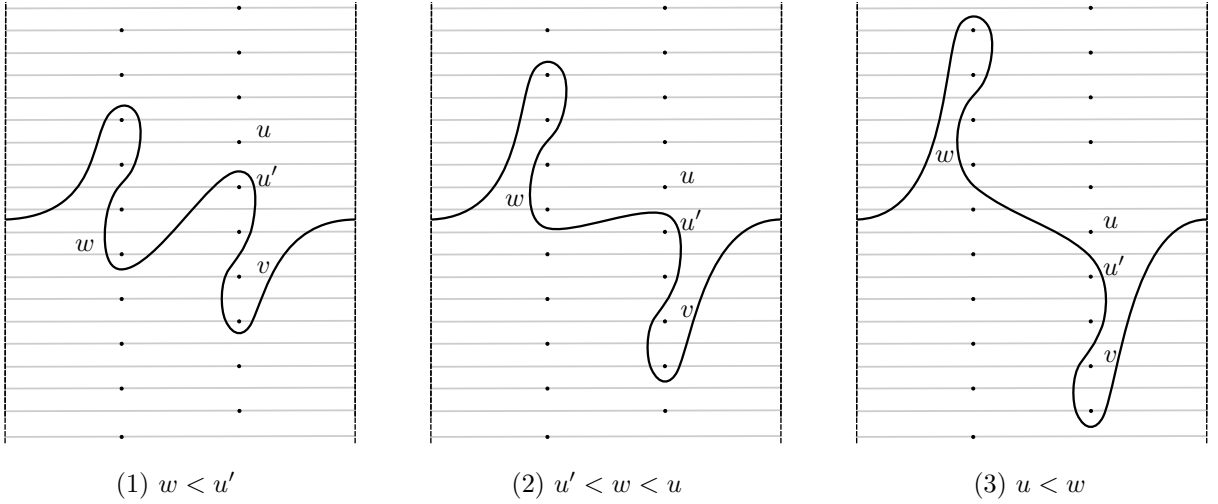

As depicted in Figure~\ref{fig:initialsecond}, we consider the following three cases: 
$$w<u', \qquad\quad  u'<w<u, \qquad\quad \text{and} \qquad\quad u<w.$$
We leave it as an easy excercise for the readers to verify that these are the exact cases that appear on \eqref{it:essential2a}, i.e.\ $w<u'$ if and only if $q<p(2\tau-1)$, $u'<w<u$ if and only if $p(2\tau -1)<q<2p\tau$, and $u<w$ if and only if $2p\tau<q$.

Suppose that $w<u'$ which corresponds to Figure~\ref{fig:initialsecond} $(1)$. By translating the pegs horizontally, we see that the image of the $p$th copy of $\eta^{--}_n$ from Step~\ref{it:step1} takes the form $\eta^{--}$ and, as in the noninitial right arc case from \eqref{it:notessential1a}, it can be checked that it spans $pn$ pegs. To be more precise, the image spans all the pegs between $u'$ and $v$ including $u'$ and excluding $v$. Following steps \ref{it:step1}, \ref{it:step2}, and \ref{it:step3}, it can be easily computed that the vertical coordinate of $u'$ minus the vertical coordinate of $v$ is $pn$. Hence, we conclude that the resulting arc is noninitial and takes the form $\eta^{--}_{pn}$.

Suppose that $u'<w<u$ which corresponds to Figure~\ref{fig:initialsecond} $(2)$. In this case, the image of the $p$th copy of $\gamma$ takes the form $\eta^{+-}$. Moreover, the image spans all the pegs between $w$ and $v$ excluding both $w$ and  $v$. As before, it can be easily computed that the vertical coordinate of $w$ minus the vertical coordinate of $v$ is $pn+p-2p\tau+q$. We conclude that the resulting arc is noninitial and takes the form $\eta^{+-}_{pn+p-2p\tau+q-1}$.

Lastly, suppose that $u<w$ which corresponds to Figure~\ref{fig:initialsecond} $(3)$. Again, the image of the $p$th copy of $\gamma$ takes the form $\eta^{+-}$. In this case, the image spans all the pegs between $u$ and $v$ excluding both $u$ and  $v$. Since the vertical coordinate of $u$ minus the vertical coordinate of $w$ is
$pn+p$, the resulting arc is noninitial and takes the form $\eta^{+-}_{pn+p-1}$. This completes the proof of the proposition for initial right arcs of the form $\eta^{--}_n$.  As mentioned above, the remaining cases can be addressed similarly.
\end{proof}

\section{Applications}\label{sec:applications}

In this section, we prove the main theorems and applications stated in the introduction. Most of these follow directly from Proposition~\ref{prop:mainbody} and \ref{prop:iteratedtorusmainbody}. We first prove Theorems~\ref{thm:main} and \ref{thm:iteratedtorus}. We prove them together, since the proofs are similar. We recall the statements.

\begin{theorem}\label{thm:mainbody}
If $K$ is a nontrivial knot, then $u(K_{p_1, q_1; p_2, q_2; \dots ; p_m, q_m}) \geq p_1 p_2 \dots p_m$. Moreover, if $1<p_0 < q_0$ and $q_0 \neq 3$, then $u(T_{p_0, q_0; p_1, q_1; \dots ; p_m, q_m}) \geq p_0 p_1 \dots p_m-1$. \end{theorem}

\begin{proof} From Proposition~\ref{prop:mainbody}, we have $\Ord(K_{p_1, q_1; p_2, q_2; \dots ; p_m, q_m}) \geq p_1 p_2 \dots p_m$. Moreover, since a positive torus knot $T_{p_0,q_0}$ is an $L$-space knot with $\Ord(T_{p_0,q_0})=p_0-1$ (see e.g.\ \cite[Corollary 5.3]{Juhasz-Miller-Zemke:2020-1}), by Proposition~\ref{prop:iteratedtorusmainbody} we have $\Ord(T_{p_0, q_0; p_1, q_1; \dots ; p_m, q_m}) \geq p_0 p_1 \dots p_m-1$. The proof is then completed by using the unknotting bound of Alishahi-Eftekhary in \eqref{eq:unknotting}.
\end{proof}

Next we prove Corollary~\ref{cor:fusionnumber}, whose statement we recall.
\begin{corollary}\label{cor:fusionnumberbody}
If $K$ is a nontrivial slice knot, then any slice disk bounded by $K_{p_1, 1; p_2, 1; \dots ; p_m, 1}$ has at least $p_1 p_2 \dots p_m+1$ local minima with respect to the radial function on $B^4$ restricted to the disk.
\end{corollary}
\begin{proof}Let $K$ be a nontrivial slice knot. By Proposition~\ref{prop:mainbody}, we have $\Ord(K_{p_1, 1; p_2, 1; \dots ; p_m, 1}) \geq p_1 p_2 \dots p_m$. Juh\'{a}sz-Miller-Zemke~\cite[Corollary 1.7]{Juhasz-Miller-Zemke:2020-1} bound the number of local minima of a slice disk from below by the torsion order, which completes the proof.
\end{proof}

Recall that for a knot $K$, its bridge index is denoted by $\br(K)$ and its braid index is denoted by $\b(K)$. We prove Corollary~\ref{cor:bridgenumber}, whose statement we recall.

\begin{corollary}\label{cor:bridgenumberbody}
Let $K$ be either an algebraic knot or an iterated torus knot $T_{p_0, q_0; p_1, q_1; \dots ; p_m, q_m}$ where $T_{p_0, q_0} \neq T_{2,3}$. If $J$ is concordant to $K$, then $\b(J) \geq \b(K)$ and $\br(J) \geq \br(K)$.
\end{corollary}
\begin{proof}First, assume that $K$ is an iterated torus knot $T_{p_0, q_0; p_1, q_1; \dots ; p_m, q_m}$ where $T_{p_0, q_0} \neq T_{2,3}$. As before, since $\Ord(T_{p_0,q_0})=p_0-1$, by Proposition~\ref{prop:iteratedtorusmainbody} we conclude that $\gamma_0$ associated to $K$ contains a right arc with length at least $p_0p_1\cdots p_m -1$. By Schubert~\cite{Schubert:1954-1}, we see that $\br(K) = \b(K) = p_0p_1\cdots p_m$ (see also \cite{Williams:1992-1}). Since $\gamma_0$ is a concordance invariant, if $J$ is concordant to $K$, then $\Ord(J) \geq p_0p_1\cdots p_m -1$. Now, the proof is complete by using the bound $\br(J) \geq \Ord(J) - 1$ \cite[Corollary 1.8]{Juhasz-Miller-Zemke:2020-1} and the fact that $\br(J) \leq \b(J)$.

Now, assume that $K$ is an algebraic knot. For this we do not need to use Proposition~\ref{prop:iteratedtorusmainbody}. Recall that if $K$ is an algebraic knot then it has the following form $T_{p_0, q_0; p_1, q_1; \dots ; p_m, q_m}$ where $q_{i+1} > p_iq_ip_{i+1}$  (see e.g.\ \cite{Eisenbud-Neumann:1985-1}). In particular, it is an $L$-space knot~\cite{Hedden:2009-1}. Using this fact, it is an easy exercise to verify that $\alpha_1 - \alpha_2 = p_0p_1\cdots p_m -1$ where \[ \Delta_{K} (t) = \sum_{k=0}^{2\ell} (-1)^k t^{\alpha_k}\]
and $\alpha_0, \ldots , \alpha_{2\ell}$ is a decreasing sequence of integers. Then by \eqref{eq:orderlspace} and the bound of \cite[Corollary 1.8]{Juhasz-Miller-Zemke:2020-1}, we conclude that $\gamma_0$ associated to $K$ contains a right arc with length $p_0p_1\cdots p_m -1$. As before, since $\gamma_0$ is a concordance invariant, the proof is complete. 
\end{proof}

\begin{remark}
The above proof shows that if $K$ is an algebraic knot or an iterated torus knot, then $\br(K) = \b(K) = \Ord(K) + 1$.  
\end{remark}

Lastly, we restate and prove Corollary~\ref{cor:genusunknotting}.

\begin{corollary}\label{cor:genusunknottingbody}
There exists a family of iterated torus knots $\{K_i\}$, such that the gap between $u(K_i)$ and $g_4(K_i)$ is arbitratily large.\end{corollary}
\begin{proof}
Let $K$ be $T_{2,-5;2,5}$. Then there is a genus 6 Seifert surface $\Sigma$ for $K$ obtained by taking two parallel copies of the genus 2 Seifert surface for $T_{2,-5}$ and connecting them with 5 half-twisted bands. Note that $T_{2,5} \# T_{2,-5}$ sits on $\Sigma$ and it bounds a genus 4 surface embedded in $\Sigma$. Since $T_{2,5} \# T_{2,-5}$ is a slice knot, if we surger along $T_{2,5} \# T_{2,-5}$ on $\Sigma$ we obtain a genus 2 surface bounded by $K$ properly embedded in $B^4$. Hence we conclude $g_4(K) \leq 2$.

In fact, we show that $g_4(K) = 2$. Note that $\nu^+(T_{2,-5})=0$ (see e.g.\ \cite{Jozsef-Celorai-Golla:2017-1}), where $\nu^+$ is the concordance invariant introduced by Hom-Wu~\cite{Hom-Wu:2016-1}. Also, they show that $\nu^+$ gives a lower bound on the 4-genus. Then by a cabling formula for $\nu^+$ of Wu~\cite[Theorem 1.1]{Wu:2016-1}, we conclude that $\nu^+(K)=2$ and $g_4(K) = 2$. 

% and later extended by Sato~\cite[Corollary 1.4]{Sato:2018-1}

Let $i$ be a positive integer and $K_i$ be $K_{i,1}$. Then we have that $g_4(K_i) = 2i$, where the upper bound is obtained by taking $n$ parallel copies of the genus 2 surface bounded by $K$, and the lower bound is given by the the inequality by Sato~\cite[Corollary 1.4]{Sato:2018-1} for $\nu^+$. Lastly, by Theorem~\ref{thm:iteratedtorus}  and the fact that $u(K) = u(-K)$, we have that $u(K_i) \geq 4i-1$. This completes the proof.
\end{proof}

\bibliographystyle{alpha}
\bibliography{bib}

\end{document}